\documentclass[10 pt]{article}
\usepackage{amssymb,amsthm}
\font\smallrm=cmr8
\font\tinyrm=cmr6
\font\bigmath=cmsy10 scaled \magstep 2
\newcommand{\emp}{\emptyset}
\newcommand{\ben}{\mathbb N}

\newcommand{\beq}{\mathbb Q}
\newcommand{\bez}{\mathbb Z}

\newcommand{\nhat}[1]{\{1,2,\ldots,#1\}}
\newcommand{\ohat}[1]{\{0,1,\ldots,#1\}}
\newcommand{\pf}{{\mathcal P}_f}
\newcommand{\mod}{\hbox{\rm mod }}

\newcommand{\supp}{\hbox{\rm supp}}
\newcommand{\subsupp}{\hbox{\smallrm supp}}
\newcommand{\bigtimes}{\hbox{\bigmath \char'2}}
\newcommand{\con}{\hbox{${}^{\frown}$}}
\newcommand{\inject}{\hbox{\vbox{\baselineskip=10 pt
\hbox{\vrule height 0 pt depth 0 pt \hskip 2 pt {\tinyrm 1-1}}
\vskip -5 pt
\hbox{$\longrightarrow$}}}}

\newtheorem{theorem}{Theorem}[section]
\newtheorem{corollary}[theorem]{Corollary}
\newtheorem{lemma}[theorem]{Lemma}
\newtheorem{question}[theorem]{Question}
\newtheorem{conjecture}[theorem]{Conjecture}
\newtheorem{proposition}[theorem]{Proposition}
\theoremstyle{definition}
\newtheorem{definition}[theorem]{Definition}

\title{Maximality of infinite partition regular matrices}
\date{}
\author{Neil Hindman
        \footnote{Department of Mathematics,
                 Howard University,
                  Washington, DC 20059, USA.\hfill\break
                  {\tt nhindman@aol.com}}
        \thanks{This author acknowledges support received from the National
                Science Foundation (USA) via Grant DMS-1160566.}
        \and
        Imre Leader
        \footnote{Department of Pure Mathematics and Mathematical Statistics,
                  Centre for Mathematical Sciences,
                  Wilberforce Road, Cambridge, CB3 0WB, UK.\hfill\break
                  {\tt i.leader@dpmms.cam.ac.uk}}
        \and
        Dona Strauss
        \footnote{Department of Pure Mathematics,
                  University of Leeds, Leeds LS2 9J2, UK.\hfill\break
                  {\tt d.strauss@hull.ac.uk}}                 
}

\begin{document}

\maketitle

\begin{abstract}
A finite or infinite matrix $A$ with rational entries (and only finitely many
non-zero entries in each row) is called {\it image partition regular} if,
whenever the natural numbers are finitely coloured, there is a vector $x$,
with entries in the natural numbers, such that $Ax$ is monochromatic. Many of
the classicial results of Ramsey theory are naturally stated in terms of
image partition regularity.

Our aim in this paper is to investigate maximality questions for image partition regular matrices. 
When is it possible to add rows on to $A$ and remain image partition regular? 
When can one add rows but `nothing new is produced'? What about adding
rows and also new variables? We prove some results about extensions of the
most interesting infinite systems, and make several conjectures. 

Perhaps our most
surprising positive result is a compatibility result for Milliken-Taylor
systems, stating that (in many cases) one may adjoin one Milliken-Taylor 
system to a translate of another and remain image partition regular. This is
in contrast to earlier results, which had suggested a strong inconsistency
between different Milliken-Taylor systems. Our main tools for this are some algebraic
properties of $\beta \ben$, the Stone-\v Cech compactification of the natural numbers.

\bigskip\noindent \textbf{Keywords:} Image partition regular; Ramsey Theory;
central sets\hfill\break
\end{abstract}

\vfill\eject

\section{Introduction}

One of the earliest theorems in Ramsey Theory is Schur's Theorem \cite{S}, which says that
if $\ben$ is finitely coloured, then there exist $x_0$ and $x_1$ such that\break
$\{x_0, x_1, x_0+x_1\}$ is monochromatic.  Some time later, van der Waerden \cite{W} proved
that whenever $\ben$ is finitely coloured and $k\in\ben$, there is a monochromatic length
$k$ arithmetic progression.  Schur's Theorem and the length $4$ version of van der Waerden's
Theorem are precisely the assertions that the following two matrices are
image partition regular.
$$\left(\begin{array}{cc} 1&0\\ 0&1\\ 1&1\end{array}\right)\hskip 30 pt
\left(\begin{array}{cc} 1&0\\ 1&1\\ 1&2\\ 1&3\end{array}\right)$$

Here we say that a matrix $A$ with rational entries, and only finitely many
non-zero entries in each row, is {\it image partition regular} or {\it IPR} if,
whenever the natural numbers are finitely coloured, there is a vector $x$,
with entries in the natural numbers, such that $Ax$ is monochromatic (meaning
that all the entries of $Ax$ are natural numbers of the same colour).

In the finite case, the IPR matrices are well understood. Roughly speaking, 
they are
the `first-entries' matrices, meaning those for which all the rows whose first
non-zero entry lies in a given column have the same entry in that column. See
Section 2 for a precise statement about this. 

[We have relegated to Section 2 background facts about finite matrices, and 
also about
the Stone-\v Cech compactification $\beta\ben$. The reader who is not 
especially interested
in such things can just skip this section and refer back to 
it when necessary.]

In the infinite case, much less is known. As a `trivial' example, note that,
given a collection of finite matrices known to be IPR, it is
possible to construct infinite IPR matrices.  For example,
if for $k\in\ben$,
$$ A_k=\left(\begin{array}{cc} 1&0\\ 1&1\\ \vdots&\vdots\\ 1&k\end{array}\right)
\hbox{ and }B=\left(\begin{array}{cccc}A_2&{\bf O}&{\bf O}&\ldots\\
{\bf O}&A_3&{\bf O}&\ldots\\
{\bf O}&{\bf O}&A_4&\ldots\\
\vdots&\vdots&\vdots&\ddots\end{array}\right)\,,$$
then $B$ is IPR (since given any finite colouring there
must be arbitrarily long arithmetic progressions in one of the colour classes, and
thus arithmetic progressions of every length in that class).
 
What is probably the first nontrivial example of an infinite IPR matrix is the {\it Finite Sums\/} matrix.
It was proved in \cite{H} that whenever $\ben$ is finitely coloured, there exists
an infinite sequence $\langle x_n\rangle_{n=0}^\infty$ such that
$FS(\langle x_n\rangle_{n=0}^\infty)$ is monochromatic, where 
$$\textstyle FS(\langle x_n\rangle_{n=0}^\infty)=\{\sum_{n\in F}\,x_n:F\in\pf(\omega)\}$$ and
$\pf(\omega)$ is the set of finite nonempty subsets of $\omega$.  
We remark that this is the assertion
that ${\bf F}$ is IPR, where all entries of ${\bf F}$ are $0$ or $1$ and
for each $i<\omega$, $\sum_{j=0}^\infty f_{i,j}2^j=i+1$.  That is,
$${\bf F}=\left(\begin{array}{cccc}1&0&0&\ldots\\
0&1&0&\ldots\\
1&1&0&\ldots\\
0&0&1&\ldots\\
1&0&1&\ldots\\
0&1&1&\ldots\\
1&1&1&\ldots\\
\vdots&\vdots&\vdots&\ddots\end{array}\right)$$
However, most of the time we will not write matrices explicitly, being content
to give the `linear system' form (as in the 
`$FS(\langle x_n\rangle_{n=0}^\infty)$' form above).

Using the Finite Sums Theorem as a tool, Milliken \cite{M} and Taylor \cite{T} independently
established the fact that each of a whole class of matrices are 
IPR.  We shall describe these matrices now.

\begin{definition}\label{defcompress} Let $k\in\omega$ and let
$\vec a=\langle a_0,a_1,\ldots,a_k\rangle$ be a sequence
in $\bez$ such that $\vec a\neq\vec 0$.  
The sequence $\vec a$ is {\it compressed\/} if and only if no $a_i=0$ and for
each $i\in\ohat{k-1}$, $a_i\neq a_{i+1}$.  The sequence $c(\vec a)=\langle c_0,c_1,\ldots, c_m\rangle$
is the compressed sequence obtained from $\vec a$ by first deleting all occurrences of $0$ and then
deleting any entry which is equal to its successor. Then $c(\vec a)$ is called the {\it compressed 
form of\/} $\vec a$.  And $\vec a$ is said to be a {\it compressed sequence\/} if $\vec a=c(\vec a)$.
\end{definition}

For example $c(\langle -2,0,-2,3,3,0,3,1,-2\rangle)=\langle -2,3,1,-2\rangle$. 
If $\vec a$ is an infinite sequence with finitely many nonzero entries, then
$c(\vec a)$ is defined analoguously, by first deleting the trailing $0$'s.

\begin{definition}\label{MTax} Let $k\in\omega$, let
$\vec a=\langle a_0,a_1,\ldots,a_k\rangle$ be a compressed sequence
in $\bez\setminus\{0\}$ with $a_k>0$, and let $\vec x=\langle x_n\rangle_{n=0}^\infty$.  
Then $MT(\vec a,\vec x)=\{\sum_{i=0}^ka_i\sum_{t\in F_i}x_t:F_0,F_1,\ldots,F_k\in \pf(\omega)
\hbox{ and }F_0<F_1<\ldots<F_k\}$, where for $F,G\in\pf(\omega)$, $F<G$ means
$\max F<\min G$. 
\end{definition}

Note that the case $\vec a=\langle 1 \rangle$ of the Milliken-Taylor theorem
(Theorem \ref{MTipr} below) is precisely the Finite Sums Theorem.

\begin{theorem}\label{MTipr}  Let $k\in\omega$ and let
$\vec a=\langle a_0,a_1,\ldots,a_k\rangle$ be a compressed sequence
in $\bez\setminus\{0\}$ with $a_k>0$.
Then whenever $\ben$ is finitely coloured there exists an infinite sequence
$\vec x=\langle x_n\rangle_{n=0}^\infty$ such that $MT(\vec a,\vec x)$ is
monochromatic.
\end{theorem}

\begin{proof} If each $a_i>0$, this is \cite[Theorem 2.2]{M} and \cite[Lemma 2.2]{T}.
The general case is a consequence of \cite[Corollary 3.6]{HLSa}.\end{proof}

In the sequel, we will occasionally need the matrix form of this. 

\begin{definition}\label{defMTM} Let $k\in\omega$, let
$\vec a=\langle a_0,a_1,\ldots,a_k\rangle$ be a compressed sequence
in $\bez\setminus\{0\}$, and let $A$ be an $\omega\times\omega$ matrix.
Then $A$ is an $MT(\vec a)$-matrix if and only if the rows of $A$ are
all rows $\vec r\in \bez^\omega$ such that $c(\vec r)=\vec a$.
The matrix $A$ is a {\it Milliken-Taylor matrix\/} if and only if
it is an $MT(\vec a)$-matrix for some $\vec a$.\end{definition}

Thus Theorem \ref{MTipr} asserts precisely that every Milliken-Taylor matrix
is IPR. It will also be convenient to use the notation $Im(\vec x)$ for the
set of the entries of a vector $\vec x$. So for example if 
$\vec a=\langle a_0,a_1,\ldots, 
a_k\rangle$ is a compressed sequence
in $\bez\setminus\{0\}$, $A$ is an $MT(\vec a)$-matrix, and 
$\vec x\in\ben^\omega$, then
$Im(A\vec x)$ is $MT(\vec a,\vec x)$.

One of the major differences between finite and infinite IPR
matrices is the following. It is a consequence of Theorem \ref{charipr}(d), and 
the fact that given any finite colouring of $\ben$, one colour class is central,
that one colour class will contain an image of each finite IPR
matrix. By way of contrast we have the following theorem of Deuber, 
Hindman, Leader and Lefmann.

\begin{theorem}\label{MTsep} Let $k, m\in\omega$ let
$\vec a=\langle a_0,a_1,\ldots,a_k\rangle$ and  
$\vec b=\langle b_0,b_1,\ldots,b_m\rangle$ be compressed sequences
in $\bez\setminus\{0\}$ with $a_k>0$ and $b_m>0$, such that 
there is no  positive rational
$r$ with $\vec a=r\vec b$, then there is a finite colouring of
$\ben$ such that there do not exist
sequences $\vec x=\langle x_n\rangle_{n=0}^\infty$ and
$\vec y=\langle y_n\rangle_{n=0}^\infty$ in $\ben$ such that 
$MT(\vec a,\vec x)\cup MT(\vec b,\vec y)$ is monochromatic.
\end{theorem}

\begin{proof} \cite[Theorem 3.1]{HLSb}.  (The proof in the case 
all entries are positive was first done in \cite[Theorems 3.2 and 3.3]{DHLL}.)
\end{proof}

In fact, if it is not the case that there is a positive rational
$r$ such that $\vec a=r\vec b$, then there is a colouring as in
Theorem \ref{MTsep} that has only
two colours.  (This can be seen in a fashion similar to the proof of
\cite[Theorem 3.14]{DHLL} where the same result is proved in the case
that all entries are positive.)

The last of the special matrices with which we will be concerned
is a {\it DH-matrix}.  Roughly speaking, this is like the Finite Sums
system, except that, instead of each $x_n$ being a fixed singleton, it can
be taken from a given finite IPR system. The DH matrices are IPR 
(see \cite{DH}), and so
for example in any finite colouring of $\ben$ one can find a sequence of
arithmetic progressions $S_1,S_2,\ldots$, with $S_i$ having length $i$, such
that all the finite sums obtained by adding up one member from each of finitely
many of the $S_i$ have the same colour.

To be precise, we shall construct such a matrix as follows.
First fix an enumeration $\langle B_n\rangle_{n=0}^\infty$ of the 
finite IPR matrices with rational entries.  
For each $n$, assume that $B_n$ is a $u(n)\times v(n)$ matrix.
For each $i\in\ben$, let $\vec 0_i$ be the $0$ vector with $i$ entries.
Let {\bf D} be an $\omega\times\omega$ matrix with all rows
of the form $\vec r_0\con\vec r_1\con\vec r_2\con\ldots$ where
each $\vec r_i$ is either $\vec 0_{v(i)}$ or is a row of $B_i$, and all but
finitely many are $\vec 0_{v(i)}$.

\begin{definition}\label{defFSS} For each $n<\omega$ let 
$Y_n\in\pf(\beq)$.  Then $FS(\langle Y_n\rangle_{n=0}^\infty)=
\{\sum_{n\in F}x_n:F\in\pf(\omega)\hbox{ and }x\in\bigtimes_{n\in F}\, Y_n\}$.
Also, for $k\in\ben$, $FS(\langle Y_n\rangle_{n=0}^k)=
\{\sum_{n\in F}x_n:\emp\neq F\subseteq\ohat{k}\hbox{ and }
x\in\bigtimes_{n\in F}\, Y_n\}$.
Given $F\in\pf(\omega)$, $\sum_{n\in F}Y_n =
\{\sum_{n\in F}x_n:x\in\bigtimes_{n\in F}\, Y_n\}$.
\end{definition}

Thus $FS(\langle Y_n\rangle_{n=0}^\infty)$ is all finite sums choosing at
most one term from each $Y_n$. 
For each $n<\omega$, let $B_n$ be the $u(n)\times v(n)$ matrix used in the
construction of ${\bf D}$.  Define $k(0)=0$ and for $n\in\omega$, let
$k(n+1)=k(n)+v(n)$.
 Assume that $\vec x\in\beq^\omega$.
For each $n\in\omega$ let $\vec y_n\in \beq^{v(n)}$ be defined
by $y_n(t)=x_{k(n)+t}$ and let $Y_n=Im(B_n\vec y_n)$.
Then $Im({\bf D}\vec x)=FS(\langle Y_n\rangle_{n=0}^\infty)$.

The plan of the paper is as follows. In Section 3 we investigate the notion
of {\it maximal} IPR matrices, meaning matrices such that no new row (not
equal to any previous row) can be 
added in such a way that the resulting matrix is IPR. 
Finite matrices cannot have this property, and neither can $\bf F$.
We observe that matrices having all rows with a given constant row
sum are maximal IPR and conjecture that these are the only examples of
 maximal IPR matrices.

We consider $\bf F$ in more detail in
Section 4, giving a more restricted sense in which it is maximal (roughly
speaking, this is the situation where we insist that the variables have
disjoint support when written out in binary or similar).
 
In Section 5 we consider image maximality.
 
\begin{definition}\label{imagedom} Let $t,u,v,w\in\ben\cup\{\omega\}$, let
$A$ be a $t\times u$ matrix and let $B$ be a $v\times w$ matrix.
Then $A$ {\it image dominates\/ }$B$ if and only if,
for each $\vec x\in \ben^u$ there exists $\vec y\in \ben^w$ 
such that $Im(B\vec y)\subseteq Im(A\vec x)$.\end{definition}

Notice that if $A$ image dominates $B$ and $A$ is IPR, 
then so is $B$.  Notice also that trivially, if
$B$ is a finite IPR matrix, then the DH-matrix
{\bf D} image dominates $B$ (because $B=B_n$ for some $n$).

We say that a matrix $A$ is {\it image maximal\/}
provided that whenever $B$ is an IPR matrix extending $A$, 
that is $B$ consists of $A$ with some rows added, then $A$ image dominates 
$B$.

We show that 
any IPR finite extension of {\bf D}
is in fact image dominated by {\bf D} itself. We conjecture that {\bf D} is
image maximal, but have been unable to show this. This is perhaps the most
tantalising of all the open questions.

Finally, in Section 6 we turn our attention to a more general notion. We say
that an IPR matrix $A$ is {\it universally image maximal\/}
provided that whenever $B$ is an IPR matrix that
image dominates $A$, then $A$ image dominates $B$. In other words, this is
like image maximality but we do not insist that $B$ is an extension of $A$.

This section contains what are perhaps our most surprising results. While we 
know that obviously {\bf F} cannot be extended to an IPR matrix by adding on
any Milliken-Taylor system except {\bf F} itself, we show that one can add on translates of 
such matrices. In some sense this ought to be impossible, in light of
Theorem \ref{MTsep}. Similarly, it `ought' to be the case that  {\bf D}
is universally image maximal, but this turns out not to be the case: one can
add a translate of `DHMT', meaning the analogue of the {\bf D} but with
for example $\langle 2,1 \rangle$ in place of $\langle 1 \rangle$. 

We do not know any examples of universally image maximal systems.
 
In this paper {\it we shall always assume that any matrix that we consider
has finitely many nonzero entries in each row\/}. We also mention briefly that
the matrices with which we
will be dealing all have countably many rows and countably many columns,
so of course the rows and columns could be rearranged so that they were
all $u\times v$ matrices for some $u,v\in\ben\cup\{\omega\}$.  But it will be convenient, given 
$\omega\times \omega$ matrices $A$ and $B$ to discuss the matrices
$$\left(\begin{array}{c} A\\ B\end{array}\right)\,,
\hskip 20 pt \left(\begin{array}{cc} A& B\end{array}\right)\hbox{,\hskip 10 pt and}\hskip 10 pt
\left(\begin{array}{cc} A&{\bf O}\\{\bf O}& B\end{array}\right)$$
where ${\bf O}$ is the $\omega\times\omega$ matrix with all zeroes.
These are respectively $(\omega+\omega)\times\omega$, $\omega\times(\omega+\omega)$,
and $(\omega+\omega)\times(\omega+\omega)$ matrices. However, we are of course always
free to relabel these as $\omega\times\omega$ matrices, and we shall often implicitly do so.

\section{Background}\label{secbackground}

In his proof of a conjecture of Rado, Deuber \cite{D} proved that certain
matrices are IPR.  (He called the set of entries in an image of
such matrices an {\it $(m,p,c)$-set\/}.  We shall have more to say about these
later.)  Deuber's matrices were special cases of {\it first entries matrices\/}.
Since the concept of a first entries matrix has not turned out to be useful
for infinite matrices, we shall restrict our definition to finite matrices.

Given a matrix we shall follow the
custom of denoting the entry in row $i$ and column $j$ by the lower case
of the upper case letter which denotes the matrix.  So the entry in row
$0$ and column $3$ of the matrix $B$ is $b_{0,3}$.

\begin{definition}\label{firstentries} Let $u,v\in\ben$ and let $A$ be a $u\times v$ matrix
with entries from $\beq$.   Then $A$ is a {\it first entries matrix\/} if and only
if no row of $A$ is $\vec 0$
and whenever $i,j\in\ohat{u-1}$ and $$\begin{array}{rcl}
k&=&\min\{t\in\ohat{v-1}:a_{i,t}\neq 0\}\\
&=&\min\{t\in\ohat{v-1}:a_{j,t}\neq 0\}\,,\end{array}$$ then $a_{i,k}=a_{j,k}>0$.  An element
$b$ of $\beq$  is a {\it first entry\/} of $A$ if and only if there is some row $i$ of
$A$ such that $b=a_{i,k}$ where $k=\min\{t\in\ohat{v-1}:a_{i,t}\neq 0\}$.
\end{definition}

A few characterisations of finite IPR matrices were found
in \cite{HL}, including two computable characterisations. Several others
have been found since.  We list in the following theorem some characterisations
that will be of interest to us in this paper. (We shall describe {\it central\/}
sets later in this section.)

\begin{theorem}\label{charipr} Let $u,v\in\ben$ and let $A$
be a $u\times v$ matrix with entries from $\beq$.   The following statements
are equivalent.
\begin{itemize}
\item[(a)] A is IPR. 
\item[(b)] There exist $m\in \ben$ and a $u\times m$ first entries matrix
$B$ with entries from $\beq$  
such that given
any $\vec y\in\ben^m$ there is some $\vec x\in\ben^v$ with
$A\vec x=B\vec y$.  
\item[(c)] There exist $m\in \ben$, a $u\times m$ first entries matrix
$E$ with entries from $\omega$, and $c\in\ben$ such that  $c$ is the only first entry 
of $E$ and given
any $\vec y\in\ben^m$ there is some $\vec x\in\ben^v$ with
$A\vec x=E\vec y$.
\item[(d)] For every central set $C$ in $\ben$, there exists
$\vec x\in\ben^v$ such that $A\vec x\in C^u$.
\item[(e)] For each $\vec {r}\in\beq^v\setminus\{\vec {0}\}$ there
exists $b\in\beq\setminus\{0\}$ such that 
$$\left(\begin{array}{c}b\vec r\\ A\end{array}\right)$$
is IPR.
\item[(f)] Whenever $C$ is a central subset of $\ben$, and $m\in\ben$,
$\{\vec x\in\ben^v:A\vec x\in C^u$ {\rm , all entries of $\vec x$ are distinct and
at least $m$ and entries of $A\vec x$ corresponding to distinct rows of 
$A$ are distinct}$\}$ is
central in $\ben^v$.
\end{itemize}\end{theorem}

\begin{proof} These are respectively statements (a), (c), (f), (h), (j), and (m) of
\cite[Theorem 15.24]{HS} except that (m) lacks the assertion that all entries of
$\vec x$ are at least $m$.  This follows because $\{\vec x\in\ben^v:$ all entries
of $\vec x$ are at least $m\}$ is an ideal of $\ben^v$ and is therefore a member
of every minimal idempotent.
\end{proof} 

Note that as a consequence of Theorem \ref{charipr}(b), first entries matrices are
IPR over $\ben$.

  As used in \cite{D}, given $m$, $p$, and $c$ in $\ben$, Deuber's 
$(m,p,c)$-set is an image of a first entries matrix with $m$ columns, all first entries
equal to $c$, all other entries from $\{-p, -p+1,\ldots,p-1,p\}$, and all 
possible rows fitting this description.  For example, a $(2,2,1)$-set is an image
of the matrix
$$\left(\begin{array}{cc}1&-2\\ 1&-1\\ 1&0\\ 1&1\\ 1&2\\ 0&1\end{array}\right)\,.$$

Most of the matrices with which we will deal will in fact have integer entries.
However some of the results about finite matrices demand that non integer entries
be allowed.  For example, if $A=\left(\begin{array}{cc}1&1\\ 1&2\end{array}\right)$
and $\vec r=\left(\begin{array}{cc} 2&1\end{array}\right)$, then $A$ is IPR
and the only $b$ such that $\left(\begin{array}{c}b\vec r\\ A\end{array}\right)$
is IPR is $b=\frac{1}{2}$, so that $b\vec r=
\left(\begin{array}{cc} 1&\frac{1}{2}\end{array}\right)$. Thus we could not include
statement (e) of Theorem \ref{charipr} if we restricted to integer entries.
(To verify that $b=\frac{1}{2}$, the reader can use \cite[Theorem 15.24(b)]{HS}.)

We conclude this section with a brief introduction to the algebraic
structure of $\beta\bez$, both under addition and multiplication.  (This
structure will be used in some proofs in the next section.) For 
proofs of the assertions made here, see \cite[Chapter 4]{HS}.

If $(S,\cdot)$ is a discrete semigroup, we take the Stone-\v Cech 
compactification $\beta S$ of $S$ to be the set of ultrafilters
on $S$, identifying the principle ultrafilters with the points
of $S$ and thereby pretending that $S\subseteq\beta S$.  (Similarly,
for example, we identify an ultrafilter $p$ on $\ben$ with the
ultrafilter $\{A\subseteq\bez:A\cap\ben\in p\}$ on $\bez$ and pretend that
$\beta\ben\subseteq\beta\bez$.) We write $S^*=\beta S\setminus S$.  So
$S^*$ is the set of nonprincipal ultrafilters on $S$.

Given a set $A\subseteq S$, $\overline A=\{p\in\beta S:A\in p\}$,
$\{\overline A:A\subseteq S\}$ is a basis for the topology on $\beta S$,
and each $\overline A$ is clopen in $\beta S$.  The operation
on $S$ is extended to $\beta S$ so that for each $p\in\beta S$
the function $q\mapsto q\cdot p$ is continuous and for each
$x\in S$ the function $q\mapsto x\cdot q$ is continuous.  Given
$p,q\in\beta S$ and $A\subseteq S$, $A\in p\cdot q$ if and only if
$\{x\in S:x^{-1}A\in q\}\in p$, where $x^{-1}A=\{y\in S:xy\in A\}$.
If the operation is denoted by $+$, one has that
$A\in p+ q$ if and only if
$\{x\in S:-x+A\in q\}\in p$, where $-x+A=\{y\in S:x+y\in A\}$.

As with any compact Hausdorff right topological semigroup, $\beta S$
has idempotents and a smallest two-sided ideal $K(\beta S)$. Idempotents
in the smallest ideal are called {\it minimal\/}.  Given an
idempotent $p\in \beta S$, $p$ is minimal if and only if
$p\beta S p$ is a group. (We shall be using this in the context
of $(\beta\ben,+)$ so that if $p$ is minimal, then 
$p+\beta\ben+p$ is a group.)

\begin{definition}\label{defcentral} Let $A\subseteq \ben$.  Then 
$A$ is {\it central\/} if and only if there is some minimal
idempotent $p$ in $(\beta\ben,+)$ such that $A\in p$.\end{definition}

\begin{definition}\label{defcenipr} Let $u,v\in\ben\cup\{\omega\}$ and let
$A$ be a $u\times v$ matrix with entries from $\beq$.
\begin{itemize}
\item[(a)] $A$ is {\it centrally IPR\/} if and
only if whenever $C$ is a central set in $\ben$, there exists
$\vec x\in \ben^v$ such that $A\vec x\in C^u$.
\item[(b)] $A$ is {\it strongly centrally IPR\/}
if and only if whenever $C$ is a central set in $\ben$, there exists
$\vec x\in \ben^v$ such that $A\vec x\in C^u$, the entries of
$\vec x$ are distinct, and entries of $A\vec x$ corresponding to distinct
rows of $A$ are distinct.\end{itemize}\end{definition}

Notice that by Theorem \ref{charipr}(f), any finite IPR
matrix is strongly centrally IPR.

\begin{theorem}\label{DHScentral} The 
matrices ${\bf D}$ and ${\bf F}$ are strongly centrally 
IPR.\end{theorem}

\begin{proof} We shall do the proof for ${\bf D}$.  The 
proof for ${\bf F}$ is similar and simpler.  In fact the
result for ${\bf F}$ is a corollary of the result for ${\bf D}$
as can be seen by restricting to those $n<\omega$ for which
$B_n$ consists of the first $v$ columns and first $2^v-1$ rows
of ${\bf F}$ for some $v$.

The proof is a modification of \cite[Theorem 16.16]{HS}.  (This is essentially
the result of the theorem of \cite{DH}, which was restricted
to $(m,p,c)$-sets.)  

Let $\langle B_n\rangle_{n=0}^\infty$, 
$\langle u(n)\rangle_{n=0}^\infty$, and $\langle v(n)\rangle_{n=0}^\infty$,
be as in the construction of ${\bf D}$.

Let $C$ be central in $\ben$ and pick a minimal idempotent $p$ in $(\beta\ben,+)$ such that 
$C\in p$.  Let $C^\star=\{x\in C:-x+C\in p\}$ and note that by 
\cite[Lemma 4.14]{HS}, if $x\in C^\star$, then $-x+C^\star\in p$.
Pick by Theorem \ref{charipr}(f) some $\vec x(0)
\in\ben^{v(0)}$ such that all entries of $B_0\vec x(0)$ are in
$C^\star$, the entries of $\vec x(0)$ are  distinct, and
entries of $B_0\vec x(0)$ corresponding to distinct rows
of $B_0$ are distinct. Let $Y_0$ be the set of entries of
$B_0\vec x(0)$.

Inductively, let $n\in\omega$\ and assume that we have  chosen
$\vec x(k)\in\ben^{v(k)}$ for each $k\in\ohat{n}$ so that,
with $Y_k$ as the set of entries of $B_k\vec x(k)$, one has
\begin{itemize}
\item[(1)] $FS(\langle Y_k\rangle_{k=0}^n)\subseteq C^\star$;
\item[(2)] the entries of $\vec x(k)$ are distinct;
\item[(3)] entries of $B_k\vec x(k)$ corresponding to distinct
rows of $B_k$ are distinct; and 
\item[(4)] if $k<n$, then 
$$\begin{array}{l}\max\big(\big\{x(k)_i:i\in\ohat{v(k)-1}\big\}\cup Y(k)\big)<{}\\
\min\big(\big\{x(k+1)_i:i\in\ohat{v(k+1)-1}\big\}\cup Y(k+1)\big)\,.\end{array}$$
\end{itemize}

Let $m=\max\big(\big\{x(n)_i:i\in\ohat{v(n)-1}\big\}\cup Y(n)\big)$
and let 
$$A=\{x\in\ben:x>m\}\cap C^\star\cap\textstyle\bigcap\{-a+C^*:a\in 
FS(\langle Y_k\rangle_{k=1}^n)\}\,.$$
Then $A\in p$ so pick by Theorem \ref{charipr}(f) some $\vec x\in\ben^{v(n+1)}$
such that all entries of $B_{n+1}\vec x(n+1)$ are in
$A$, the entries of $\vec x(n+1)$ are  distinct and all at least $m+1$, and
entries of $B_{n+1}\vec x(n+1)$ corresponding to distinct rows
of $B_{n+1}$ are distinct. Let $Y_{n+1}$ be the set of entries of
$B_{n+1}\vec x(n+1)$.
Then $FS(\langle Y_k\rangle_{k=0}^{n+1})\subseteq C^\star$.
\end{proof}

\section{Extending the Finite Sums matrix}

We are concerned in this section with the general question, given 
an IPR matrix $A$, which matrices $B$ can be added
so that $\left(\begin{array}{c}B\\ A\end{array}\right)$ is IPR.  
We saw in Theorem \ref{charipr}(e) that if $A$ is finite, it
can be extended one row at a time practically at will.

By way of contrast, there exist finite {\it kernel partition regular\/}
matrices which cannot be extended at all.  (A $u\times v$ matrix 
$A$ is kernel partition regular if and only if whenever $\ben$ is finitely
coloured, there exists $\vec x\in\ben^v$ whose entries are monochromatic
such that $A\vec x=\vec 0$.)  Consider the matrix $A=\left(\begin{array}{ccc}
1&1&-1\end{array}\right)$.  The assertion that $A$ is kernel partition regular
is Schur's Theorem. The only way $A$ can be extended is by essentially
repeating the same equation. That is, if $\left(\begin{array}{ccc}u&v&w\\ 1&1&-1\end{array}\right)$
is kernel partition regular, then $u=v=-w$.  (This can be seen by invoking
Rado's Theorem \cite[Satz IV]{R}, or by noting that, if $u\neq v$ and 
$\alpha=-\frac{u+w}{v+w}$, then $\alpha>0$ and $\alpha\neq 1$.  Then
colour $\ben$ in two colours so that for any $x\in \ben$, if $\alpha x\in\ben$,
it has a different colour.)

\begin{definition}\label{maxipr} A matrix $A$ is {\it maximal IPR\/} provided it
is IPR and if $\vec r$ is a row with finitely many nonzero entries which is
not a row of $A$, then $\left(\begin{array}{c}\vec r\\ A\end{array}\right)$
is not IPR.\end{definition}

We give a trivial example of a maximal finite sums matrix in the following proposition.

\begin{proposition}\label{maximal} Let $c$ be a positive rational number, and let $A$ denote an
$\omega\times\omega$ matrix over $\beq$ which contains all possible rows whose entries have a 
sum equal to $c$. Then $A$ is maximal IPR. \end{proposition}

\begin{proof} We first observe that if $\vec x\in\ben^{\omega}$ and $A\vec x\in\ben^{\omega}$, 
then $\vec x$ has constant entries. To see this, let $m$ and $n$ be distinct elements of $\omega$
and pick $r\in\ben$ such that $r>cx_m$.
The vector in $\beq^{\omega}$ 
whose $m$'th entry is $c+r$ and whose $n$'th entry is $-r$, with all  other 
entries being 0, is a row of $A$. So $(c+r)x_m>rx_n$ and hence $x_m\geq x_n$. 
By symmetry, $x_n\geq x_m$ and so $x_m=x_n$.

Now suppose that the sum of the entries of $\vec r$ is $b\neq c$. We can define a 
finite colouring of $\beq^+$ such that, for every $s\in\beq^+$, $bs$ and $cs$ have 
different colours. 
It follows that $\left(\begin{array}{c}\vec r\\ A\end{array}\right)$ cannot be IPR over $\ben$. 
For example, observe that every element of $\beq^+$ has a unique decomposition 
of the form $\prod_{i\in \ben}\,p_i^{k_i}$ where $(p_i)_{i\in\ben}$ denotes the 
sequence of prime numbers and each $k_i\in\bez$. We can choose a prime $p$ which 
occurs with different exponents $i$ and $j$ in the decomposition of $b$ and $c$ respectively. We can choose a prime
$q>\max(|i|,|j|)$ and colour each $s\in \beq^+$ by the value (mod $q$) of the exponent of $p$ in 
the prime decomposition of $s$. \end{proof}

\begin{conjecture} There are no maximal IPR matrices other than those given by
Proposition {\rm \ref{maximal}}.
\end{conjecture}
 
The reason for the title of the section is that the only results we have
on the general question deal with extending the Finite Sums matrix. 
(Recall that we are denoting the Finite Sums matrix by ${\bf F}$.)
Thus, we are addressing the question of
which matrices $B$ (of dimension $u\times\omega$ for some $u\in\ben\cup\{\omega\}$)
have the property that $\left(\begin{array}{c}B\\ {\bf F}\end{array}\right)$
is IPR.  In the case that $u$ is finite, we can answer that
question completely.  (Recall that we are assuming that all the matrices which we consider
have finitely many nonzero entries in each row, so that if $u$ is finite,
then $B=\left(\begin{array}{cc}A&{\bf O}\end{array}\right)$ where $A$ is some finite
matrix with $u$ rows and ${\bf O}$
is the $u\times \omega$ matrix with all zeroes.)

\begin{theorem}\label{addfinite} Let $u,v\in\ben$, let $A$ be a $u\times v$ matrix
with rational entries, let ${\bf F}_v$ consist of the first $v$ columns
and the first $2^v-1$ rows of ${\bf F}$, and let ${\bf O}$ be the 
$u\times \omega$ matrix with all zeroes.  The following statements are equivalent.
\begin{itemize}
\item[(a)] $\left(\hskip -8 pt\begin{array}{c}\begin{array}{cc}A&{\bf O}\end{array}\\
{\bf F}\end{array}\hskip -8 pt\right)$ is strongly centrally IPR.
\item[(b)] $\left(\hskip -8 pt\begin{array}{c}\begin{array}{cc}A&{\bf O}\end{array}\\
{\bf F}\end{array}\hskip -8 pt\right)$ is centrally IPR.
\item[(c)] $\left(\hskip -8 pt\begin{array}{c}\begin{array}{cc}A&{\bf O}\end{array}\\
{\bf F}\end{array}\hskip -8 pt\right)$ is IPR.
\item[(d)] $\left(\begin{array}{c}A\\ {\bf F}_v\end{array}\right)$ is IPR.
\end{itemize}\end{theorem}

\begin{proof} The only nontrivial implication is that (d) implies (a), so
assume that $\left(\begin{array}{c}A\\ {\bf F}_v\end{array}\right)$ is IPR.  
Let $C$ be a central subset of $\ben$ and pick a minimal idempotent
$p\in\beta\ben$ such that $C\in p$.  Let $C^\star=\{x\in C:-x+C\in p\}$ and note
that, by \cite[Lemma 4.14]{HS}, if $x\in C^\star$, then $-x+C^\star\in p$. Then
$C^\star$ is central, so by Theorem \ref{charipr}(f),
pick $x_0,x_1,\ldots,x_{v-1}$, all distinct, such that
$$\left(\begin{array}{c}A\\ {\bf F}_v\end{array}\right)\left(\begin{array}{c}
x_0\\ x_1\\ \vdots\\ x_{v-1}\end{array}\right)\in (C^\star)^u$$ and entries corresponding to 
distinct rows of $\left(\begin{array}{c}A\\ {\bf F}_v\end{array}\right)$
are distinct.  Let $m$ be the maximum of all of these entries.
Let 
$$B=\{x\in\ben:x>m\}\cap\textstyle\bigcap\{-a+C^\star:a\hbox{ is an entry of }
\left(\begin{array}{c}A\\ {\bf F}_v\end{array}\right)\left(\begin{array}{c}
x_0\\ x_1\\ \vdots\\ x_{v-1}\end{array}\right)\}\,.$$
Then $B\in p$ so by \cite[Theorem 5.14]{HS}, pick a sequence
$\langle H_n\rangle_{n=0}^\infty$ in $\pf(\omega)$ such that
for every $n\in\omega$, $\max H_n<\min H_{n+1}$ and, 
if $y_n=\sum_{t\in H_n}\,2^t$, then $FS(\langle y_n\rangle_{n=0}^\infty)\subseteq B$.
By discarding a few terms, we may assume that $\min H_0\geq m$.
For $n\geq v$, let $x_n=y_n$.  Then all entries of
$\left(\hskip -8 pt\begin{array}{c}\begin{array}{cc}A&{\bf O}\end{array}\\
{\bf F}\end{array}\hskip -8 pt\right)\vec x$ are in $C$, entries of $\vec x$ are distinct,
and entries of  $\left(\hskip -8 pt\begin{array}{c}\begin{array}{cc}A&{\bf O}\end{array}\\
{\bf F}\end{array}\hskip -8 pt\right)\vec x$ corresponding to distinct
rows of $\left(\hskip -8 pt\begin{array}{c}\begin{array}{cc}A&{\bf O}\end{array}\\
{\bf F}\end{array}\hskip -8 pt\right)$ are distinct.
\end{proof}

The above proof in fact establishes something stronger than statement (a).  For
example, let
$A$ be an $\omega\times \omega$ matrix with all rows beginning with $1$ and then $2$
and followed by $0$'s and $1$'s with finitely many $1$'s.  The proof shows
that $\left(\begin{array}{c}A\\ {\bf F}\end{array}\right)$ is strongly
centrally IPR.

We do not know of any matrices that have entries not equal to either
$0$ or $1$ arbitrarily far to the right and extend the Finite Sums matrix.
We strongly suspect that the answer to the following question is ``no", but
cannot prove that it is.

\begin{question}\label{onetwoone} Let
$$B=\left(\begin{array}{cccccc}
1&2&1&0&0&\ldots\\
0&1&2&1&0&\ldots\\
0&0&1&2&1&\ldots\\
\vdots&\vdots&\vdots&\vdots&\vdots&\ddots\end{array}\right)\,.$$
Is $\left(\begin{array}{c}B\\ {\bf F}\end{array}\right)$ IPR?\end{question}

In the light of the following theorem, the matrix defined in Question \ref{onetwoone}
is the simplest possible matrix of this kind about which the question arises.
In this theorem we let ${\bf F'}$ be the submatrix of ${\bf F}$ consisting of
the rows with at most two $1$'s.

\begin{theorem}\label{extendingF} Let $k\in\ben\setminus\{1\}$ and let $a_0,a_1,a_2,\ldots,a_{k-1}\in \bez$,
with $a_0$ and $a_{k-1}$ being non-zero. Let $A$ denote the $\omega\times\omega$ matrix
whose $n$'th  row has entries $a_0,a_1,\ldots,a_{k-1}$ in the columns indexed by $n,n+1,n+2,\ldots,n+k-1$ 
respectively, with all other entries being zero. Assume that
$B=\left(\begin{array}{c}A \\ {\bf F'}\end{array}\right)$ is IPR over $\ben$. Then $a_0=a_{k-1}=1$. \end{theorem}

\begin{proof} Let $p$ be a prime number satisfying $p>\sum_{i=0}^{k-1}\,|a_i|$. 
Every $x\in\ben$ can be expressed uniquely as $x=\sum_{n=0}^{\infty}e_n(x)p^n$, where each 
$e_n(x)\in\{0,1,2,\ldots,p-1\}$ and only finitely many are nonzero.
We let $\supp(x)=\{n\in\omega:e_n(x)\neq 0\}$, let $m(x)=\min\supp(x)$, and let
$M(x)=\max\supp(x)$.

We define a finite colouring $\psi$ of $\ben$, agreeing that $\psi(x)=\psi(y)$ if and only
if 
\begin{itemize}
\item[(1)] $e_{m(x)}(x)=e_{m(y)}(y)$,
\item[(2)] $e_{M(x)}(x)=e_{M(y)}(y)$,
\item[(3)] $e_{M(x)-1}(x)=e_{M(y)-1}(y)$, and
\item[(4)] $M(x)\equiv M(y)$ (mod $3$),
\end{itemize}
Let $\vec x\in\ben^{\omega}$ be a vector for which the entries of
$B\vec x$ are monochromatic. Let $a$, $b$, and $c$ be the fixed
values of $e_{m(x_n)}(x_n)$, $e_{M(x_n)}(x_n)$, and $e_{M(x_n)-1}(x_n)$ respectively, for $n\in\omega$.

Let $m$ and $n$ be distinct elements of $\omega$.  Then $m(x_m)\neq m(x_n)$ because $x_m$,
$x_n$, and $x_n+x_m$ are all entries of $B\vec x$ and 
$2a\not\equiv a$ (mod $p$). Now assume that $a_0\neq 1$. 
Choose $r_0\in\ohat{k-1}$ such that
$$m(x_{r_0})=\min\{m(x_0),m(x_1),\ldots, 
m(x_{k-1})\}\,.$$ 
Then $m(a_0x_0+a_1x_1+\ldots+a_{k-1}x_{k-1})=m(x_{r_0})$.
Consequently, $m(x_{r_0})<m(x_0)$ because 
$a\not\equiv a_0a$ (mod $p$). 
Similarly, if $r_1\in\{r_0,r_0+1,\ldots,r_0+k-1\}$ such that
$m(x_{r_1})=\min\{m(x_{r_0}),m(x_{r_0+1}),\ldots,m(x_{r_0+k-1})\}$,
then $m(x_{r_1}) 
<m(x_{r_0})$.
 Proceeding in this way, we can define an infinite
decreasing sequence in $\omega$, which is impossible. So $a_0=1$.

We now claim that $M(x_m)\neq M(x_n)$. If $M(x_m)=M(x_n)$,
  then $M(x_m)\leq M(x_m+x_n)\leq
 M(x_m)+1$. This implies that $M(x_m+x_n)=M(x_m)$ and hence that
$b<\frac{p}{2}$. So the most significant digit in the base $p$ expansion of 
$x_m+x_n$ is $2b$ or $2b+1$, and this cannot be equal to $b$, a contradiction.

We observe that, if $x_m<x_n$ and $M(x_n)=s$, then $M(x_m)\leq s-3$. So
$x_n\geq p^s$ and $x_m<p^{s-2}$, and hence $\frac{x_n}{x_m}>p^2$.

Assume that $a_{k-1}\neq 1$.  Pick the first $n\geq k-1$ such that 
$$M(x_n)>\max\{M(x_0),M(x_1),\ldots,M(x_{k-2})\}\,.$$ Then
$x_n=\max\{x_{n-k+1},x_{n-k+2},\ldots,x_n\}$.  
Let $t=\sum_{i=0}^{k-2}\,a_ix_{n-k+1+i}$.
Then $$\textstyle |t|<(\sum_{i=0}^{k-2}\,|a_i|)\frac{x_n}{p^2}<x_n\leq |a_{k-1}|x_n\,.$$
Since $a_{k-1}x_n+t>0$, we must have $a_{k-1}>0$ and hence $a_{k-1}\geq 2$.
Let $r=M(x_n)$. We have observed that, if $x_i<x_n$, then $x_i<p^{r-2}$.
So $|t|<p^{r-1}$. We have that $M(a_{k-1}x_n+t)=r$, because $p^{r-1}< p^r-p^{r-1}<a_{k-1}x_n+t
<p^{r+2}+p^{r-1}<p^{r+3}$. Therefore $a_{k-1}x_n+t=bp^r+cp^{r-1}+u$, where $0\leq u<p^{r-1}$. 
We also have $x_n=bp^r+cp^{r-1}+v$, where $0\leq v<p^{r-1}$.
So $p^r\leq x_n\leq a_{k-1}x_n-x_n=u-v-t<2p^{r-1}$, a contradiction.\end{proof}

As we saw in Theorem \ref{MTsep}, if $k\in\ben$, 
$\vec a=\langle a_0,a_1,\ldots,a_k\rangle$ and 
$M$ is an $MT(\vec a)$ matrix, then 
$\left(\begin{array}{cc} M&{\bf O}\\{\bf O}& {\bf F}\end{array}\right)$
is not IPR.  We mention that we shall see, 
in Section 6, that 
$\left(\begin{array}{cc} \overline 1&M\\ \overline 0& {\bf F}\end{array}\right)$
is partition regular, where $\overline 1$ and $\overline 0$ are the constant
length $\omega$ column vectors.  That is, given any finite colouring of $\ben$, there
must exist a sequence $\vec x=\langle x_n\rangle_{n=0}^\infty$ and $b\in\ben$
such that $FS(\vec x)\cup \big(b+MT(\vec a,\vec x)\big)$ is monochromatic.

\section{A maximal property of the Finite Sums matrix}

In this section we show that that the Finite Sums matrix, ${\bf F}$, is
maximal with respect to a particular notion of image partition regularity.

\begin{definition}\label{rapidipr} Let $u,v\in\ben\cup\{\omega\}$ and let
$A$ be a $u\times v$ matrix with entries from $\beq$. $A$ is {\it rapidly IPR\/} if and only if
whenever $\ben$ is finitely coloured and $p$ is a prime, there exists $\vec x\in \ben^v$ such 
that the entries of $A\vec x$ are monochromatic and whenever $i+1<v$ and $s\in\omega$,
if $p^s\leq x_i$, then $p^{s+8}$ divides $x_{i+1}$.
\end{definition}

We observe that ${\bf F}$, indeed all Milliken-Taylor matrices with final coefficient positive,  are rapidly IPR. To see this,
suppose that  $M$ is a Milliken-Taylor matrix determined by the compressed sequence
$\vec a=\langle a_1,a_2,\ldots,a_k\rangle$ in $\bez$ where $a_k>0$, let $p$ be a prime, and
let $\ben$ be finitely coloured.  Let $q$ be an idempotent in
$\beta\ben$.  Define $f:\ben^k\to\bez$ by
$f(x_1,x_2,\ldots,x_k)=a_1x_1+a_2x_2+\ldots+a_kx_k$ and
define $h:\bigcup_{m=1}^\infty\ben^m\to q$ as follows.  If
$(x_1,x_2,\ldots,x_m)\in\ben^m$ and $s=\max\{t\in\omega:p^t\leq x_m\}$,
then $h(x_1,x_2,\ldots,x_m)=p^{s+8}\ben$.  (By \cite[Lemma 6.6]{HS}, 
$p^{s+8}\ben\in q$.) Then by \cite[Theorem 3.3]{HLSa} one may choose
$\langle x_t\rangle_{t=0}^\infty$ as required.

In particular, since a Milliken-Taylor matrix determined by the compressed sequence
$\vec a=\langle a_1,a_2,\ldots,a_k\rangle$ with $k>1$ is not centrally IPR by
Theorem \ref{MTsep}, we see that rapidly IPR matrices need not by centrally IPR.
On the other hand, the matrix $A=\left(\begin{array}{cc}0&1\\ 1&2\end{array}\right)$
is strongly centrally IPR, since it is a first entries matrix, but is not
rapidly IPR. To see the latter assertion, colour $x\in\ben$ by whether
$\max\{t\in\omega:2^t\leq x\}$ is even or odd and let $p=2$.

We shall show in Theorem \ref{notrapidipr} that ${\bf F}$ is maximal among
rapidly IPR matrices with integer entries.  To do this
we will utilize the representation of integers to negative bases, as was done
in \cite{HLSb}.

We omit the routine proof of the following lemma.

\begin{lemma}\label{negbase} Let $p\in\ben\setminus\{1\}$, let $s\in\omega$, and
let $x\in\bez\setminus\{0\}$. There exist $\langle d_i\rangle_{i=0}^s$ with each $d_i\in\ohat{p-1}$ and
$d_s>0$ such that $x=\sum_{i=0}^sd_i(-p)^i$ if and only if  
\begin{itemize}
\item[(1)] $s$ is even and $\frac{p^s+p}{p+1}\leq x\leq\frac{p^{s+2}-1}{p+1}$ or
\item[(2)] $s$ is odd and $\frac{-p^{s+2}+p}{p+1}\leq x\leq\frac{-p^s-1}{p+1}$.
\end{itemize}
\end{lemma}

It follows immediately from Lemma \ref{negbase} that given $p\in\ben\setminus\{1\}$ and
$x\in\bez$, there is a unique choice of $\langle d_i\rangle_{i=0}^\infty$ with each
$d_i\in\ohat{p-1}$ such that $x=\sum_{i=0}^\infty d_i(-p)^i$.  In the following
definition we suppress the dependence of $d_i(x)$ and $\supp(x)$ on $p$ because we will
be using only one value of $p$ in the proof of Theorem \ref{notrapidipr}.

\begin{definition}\label{defdandsupp} Let $x\in\bez$ and let $p\in\ben\setminus\{1\}$.
\begin{itemize}
\item[(a)] $\langle d_i(x)\rangle_{i=0}^\infty$ is the unique sequence in $\ohat{p-1}$
such that\hfill  $x=\sum_{i=0}^\infty d_i(x)(-p)^i$.
\item[(b)] $\supp(x)=\{i\in\omega:d_i(x)\neq 0\}$.
\end{itemize}\end{definition}

\begin{lemma}\label{suppax} Let $x\in\ben$, let $a\in\bez\setminus\{0\}$,
and let $p\in\ben$ with $p>|a|$.  Let $s=\max\ \supp(x)$ and let
$r=\max\ \supp(ax)$. 
\begin{itemize}
\item[(1)] $p^{s-2}<x<p^{s+1}$.
\item[(2)] If $a>0$, then $s\leq r\leq s+2$.
\item[(3)] If $a<0$, then $s-1\leq r\leq s+1$.
\end{itemize}
\end{lemma}

\begin{proof} We have by Lemma \ref{negbase} that

\medskip
\hbox to\hsize{($*$)\hfill $\displaystyle\frac{p^s+p}{p+1}\leq x\leq\frac{p^{s+2}-1}{p+1}$.\hfill}
\medskip

\noindent
Conclusion (1) then follows immediately.
Conclusions (2) and (3) are derived in the same way.  
We will do the computations for (3), since they are
slightly more complicated. 

So assume $a<0$. By ($*$) we have $a\frac{p^{s+2}-1}{p+1}\leq ax\leq a\frac{p^s+p}{p+1}$
and by Lemma \ref{negbase} we have that
$\frac{-p^{r+2}+p}{p+1}\leq ax\leq\frac{-p^r-1}{p+1}$.
Thus we have that $a(p^{s+2}-1)\leq -p^r-1$ and
$-p^{r+2}+p\leq a(p^s+p)$.  Consequently $p^r+1\leq |a|(p^{s+2}-1)<p^{s+3}-p$
and $p^s+p\leq |a|(p^s+p)\leq p^{r+2}-p$. Since $p^r+p+1<p^{s+3}$, we have that
$r<s+3$ so, since $r$ is odd, $r\leq s+1$.  Since
$p^s+2p\leq p^{r+2}$, $s<r+2$ so $r\geq s-1$.
\end{proof}

In the following theorem we will show that one cannot add any row $\vec r$ to
${\bf F}$ whose nonzero entries in order are $a_1,a_2,\ldots,a_k$ and remain
rapidly IPR unless
$a_1=a_2=\ldots=a_k=1$ (in which case $\vec r$ is already a row of ${\bf F}$).
By way of contrast, by Theorem \ref{addfinite}, if any $a_i=1$, then
$\left(\begin{array}{c}\vec r\\ {\bf F}\end{array}\right)$ is strongly
centrally IPR (because the columns can be rearranged so that
$\left(\begin{array}{c}\vec r\\ {\bf F}\end{array}\right)$ extends a finite
first entries matrix).

\begin{theorem}\label{notrapidipr} The Finite Sums matrix ${\bf F}$ is maximal among
rapidly IPR matrices with integer entries.\end{theorem}

\begin{proof} Suppose not and let $\vec r\in\bez^\omega$ with finitely many nonzero
entries and not all entries in $\{0,1\}$ such that $B=\left(\begin{array}{c} \vec r\\
{\bf F}\end{array}\right)$ is rapidly IPR.  Assume that the nonzero entries
of $\vec r$ are $a_1,a_2,\ldots,a_k$ in order and that they occur in columns
$j(1),j(2),\ldots,j(k)$ respectively.  Let $r=\min\{i\in\nhat{k}: a_i\neq 0\}$.
Pick a prime $p$ such that $k<p$ and for all $i\in\nhat{k}$, $2|a_i|<p$.

For $x\in\bez\setminus\{0\}$, define $f(x)=d_{\min\ \subsupp(x)}(x)$, the
least significant digit of $x$ in the base $-p$ expansion.
For $x\in\bez\setminus\{0\}$ with $\min\ \supp(x)=s\geq 3$, define
$\phi(x)=\langle u_0,u_1,u_2,u_3\rangle\in\nhat{p-1}\times\ohat{p-1}^3$,
where for $i\in\{0,1,2,3\}$, $u_i=d_{s-i}(x)$.

For $x\in\bez\setminus\{0\}$ and 
$\langle v,u_0,u_1,u_2,u_3\rangle\in\nhat{p-1}^2\times\ohat{p-1}^3$,
let 
$$\begin{array}{rl}G_{v,u_0,u_1,u_2,u_3}(x)=\{(s,t):& s\in 2\ben\,,\,t\in\ben\,,\,t>s+3\,,\,d_t(x)=v\,,\\
&\hbox{for }i\in \{0,1,2,3\}\,,\,d_{s-i}(x)=u_i\,,\\
&\hbox{and for }s<i<t\,,\,d_i(x)=0\}\,.\end{array}$$ 
Thus $G_{v,u_0,u_1,u_2,u_3}(x)$ is the set of ``gaps'' of the form
$v0\ldots0u_0u_1u_2u_3$ with $u_0$ in even position and at least three $0$'s between $v$ and
$u_0$, occurring in the base $-p$ expansion of $x$ written with the most
significant digit on the left.  Define $\psi_{v,u_0,u_1,u_2,u_3}(x)\in\ohat{p-1}$
by 
$$\psi_{v,u_0,u_1,u_2,u_3}(x)\equiv |G_{v,u_0,u_1,u_2,u_3}(x)|\ (\mod p)\,.$$

Let $\theta$ be a finite colouring of $\ben$ such that one colour class
is $\nhat{p^4}$ and for $x,y\in\ben\setminus\nhat{p^4}$, $\theta(x)=\theta(y)$ 
if and only if
\begin{itemize}
\item[(1)] $\phi(x)=\phi(y)$;
\item[(2)] $f(x)=f(y)$; and
\item[(3)] for all $\langle v,u_0,u_1,u_2,u_3\rangle\in\nhat{p-1}^2\times\ohat{p-1}^3$
and all\hfill\break
 $i\in\nhat{k}$, $\psi_{v,u_0,u_1,u_2,u_3}(a_ix)=\psi_{v,u_0,u_1,u_2,u_3}(a_iy)$.
\end{itemize}
Pick $\vec x\in\ben^\omega$ such that $B\vec x$ is monochromatic with respect to
$\theta$ and for all $t,s\in \omega$, if $p^s\leq x_t$, then $p^{s+8}$ divides $x_{t+1}$.

We note that for all $i,j\in\nhat{k}$ and all $t<\omega$,
$\max\ \supp(a_ix_t)+3<\min\ \supp(a_jx_{t+1})$. To see this,
let $s=\max\ \supp(x_t)$. Then by Lemma \ref{suppax}(1), 
$x_t>p^{s-2}$ so $p^{s+6}$ divides $x_{t+1}$, and thus
$\min\ \supp(a_jx_{t+1})=\min\ \supp(x_{t+1})
>s+5\geq\max\ \supp(a_ix_t)+3$, where the last inequality holds
by Lemma \ref{suppax}(2) or (3).

Let $\langle u_0,u_1,u_2,u_3\rangle=\phi(x_0)$, the constant value 
of $\phi$ on the entries of $B\vec x$.  Let 
$v=f(a_rx_0)$.  (If $w$ is the constant value of $f$ on the
entries of $B\vec x$, then $v\equiv a_rw\ (\mod p)$.)
Let $\psi=\psi_{v,u_0,u_1,u_2,u_3}$.

\begin{lemma}\label{phiax} Let $x\in\ben$ with $x>p^4$ and
assume that $\phi(x)=\langle u_0,u_1,u_2,u_3\rangle$.
If $a\in\ben$ with $1<a<\frac{p}{2}$, then $\phi(ax)\neq\langle u_0,u_1,u_2,u_3\rangle$.
\end{lemma}

\begin{proof}  Suppose that $\phi(ax)=\langle u_0,u_1,u_2,u_3\rangle$. 
Then the four most significant digits in the base $-p$ expansion of $x$ and $ax$ are the 
same so there exists $m\in\omega$ such that
$\max\ \supp( p^{2m}x))=\max\ \supp(ax))=s$, say. So we have $p^{2m}x=y+z$
and $ax=w+z$ for some $y,z,w\in \bez$ satisfying $\max\ \supp(y))\leq s-4$ and
$\max\ \supp(w))\leq s-4$. It follows from Lemma \ref{suppax}(1)  that $|p^{2m}-a|x<2p^{s-3}$.
Since $|p^{2m}-a|\geq 1$ we have that $x<2p^{s-3}$ so that $ax<p^{s-2}$, contradicting 
Lemma  \ref{suppax}(1).
\end{proof}

\begin{lemma}\label{psisum} Let $x\in\ben$ with $x>p^4$, let 
$y\in\bez\setminus\{0\}$ such that $\max\ \supp(x)+5<\min\ \supp(y)$,
and let $i\in\nhat{k}$.  If $\phi(x)=\langle u_0,u_1,u_2,u_3\rangle$, 
then $\mod p$
$$\psi(a_ix+y)\equiv\left\{\begin{array}{ll}\psi(a_ix)+\psi(y)+1&\hbox{\rm if }a_i=1\hbox{\rm\ and }f(y)=v\\
\psi(a_ix)+\psi(y)&\hbox{\rm otherwise}\end{array}\right.$$
\end{lemma}

\begin{proof} $G_{v,u_0,u_1,u_2,u_3}(a_ix+y)=G_{v,u_0,u_1,u_2,u_3}(a_ix)\cup G_{v,u_0,u_1,u_2,u_3}(y)
\cup H$, where $H =\big\{\big(\max\ \supp(a_ix),\min\ \supp(y)\big)\big\}\hbox{ if }
a_ix>0, \phi(a_ix)=\langle u_0,u_1,u_2,u_3\rangle\hbox{, and}\break f(y)=v$, and 
$H=\emp$ otherwise.  If $a_i<0$, then $a_ix<0$, and  by Lemma \ref{phiax}, if $a_i>1$, then
$\phi(a_ix)\neq\langle u_0,u_1,u_2,u_3\rangle$.
\end{proof}

Since $f(x_0)=f(a_1x_{j(1)}+a_2x_{j(2)}+\ldots+a_kx_{j(k)})=f(a_1x_{j(1)})$, we have that
$a_1=1$ so $r>1$. 

Given any $i\in\nhat{k}$ and any $j<\omega$, $\psi(a_ix_j+a_ix_{j+1})=
\psi(a_ix_j)=\psi(a_ix_{j+1})$ since $x_j$, $x_{j+1}$ and $x_j+x_{j+1}$ are all
entries of $B\vec x$.  Also, either
\begin{itemize}
\item[(1)] $a_i\neq 1$ in which case either $a_ix_j<0$ or, 
by Lemma \ref{phiax}, $\phi(a_ix_j)\neq\langle u_0,u_1,u_2,u_3\rangle$, or 
\item[(2)] $a_i\neq r$
in which case $f(a_ix_{j+1})\neq v$.  
\end{itemize} Therefore by Lemma \ref{psisum},
$\psi(a_ix_j+a_ix_{j+1})=\psi(a_ix_j)+\psi(a_ix_{j+1})$ so that $\psi(a_ix_j)=0$.

By repeated applications of Lemma \ref{psisum}, beginning with $\psi(a_{k-1}x_{j(k-1)}+a_kx_{j(k)})$,
we see that $\psi(a_1x_{j(1)}+a_2x_{j(2)}+\ldots+a_kx_{j(k)})$ is the number of $i\in\nhat{k-1}$ for 
which $a_i=1$ and $a_{i+1}=r$.  Since this number is at least $1$ and less than $p$,
we have that $\psi(a_1x_{j(1)}+a_2x_{j(2)}+\ldots+a_kx_{j(k)})\neq\psi(x_0)$, a contradiction.
\end{proof}

\section{Image domination and image maximality}

We shall say that a matrix $A$ is {\it image maximal\/}
provided that whenever $B$ is an IPR matrix extending $A$, that
is $B$ consists of $A$ with some rows added, then $A$ image dominates $B$.

We note that the Finite Sums matrix ${\bf F}$ is not image maximal.  Indeed,
Let $B$ be ${\bf F}$ with the row $\left(\begin{array}{ccccc}1&2&0&0&\ldots\end{array}\right)$
added. By Theorem \ref{addfinite}, $B$ is IPR because $\left(\begin{array}{c}1\hskip 10 pt 2\\ {\bf F}_2\end{array}\right)$
is a first entries matrix.
For $n\in\omega$, let $x_n=2^{2n}$.  Then
$FS(\langle 2^{2n}\rangle_{n=0}^\infty)=Im({\bf F}\vec x)$ and 
$FS(\langle 2^{2n}\rangle_{n=0}^\infty)$ contains no image of $B$.  (One
cannot have $\{y_0,y_1,y_0+y_1,y_0+2y_1\}\subseteq FS(\langle 2^{2n}\rangle_{n=0}^\infty)$.)

We show now that the DH-matrix ${\bf D}$ is {\it finitely image maximal} in 
the sense that any IPR extension of ${\bf D}$ obtained by adding finitely many rows is
image dominated by ${\bf D}$.

\begin{theorem}\label{DHfmax} Let $m\in\ben$ and let $C$ be an $m\times\omega$ matrix
such that $A=\left(\begin{array}{c} C\\{\bf D}\end{array}\right)$ is IPR.  
Then ${\bf D}$ image dominates $A$.\end{theorem}

\begin{proof} Let $\langle B_n\rangle_{n=0}^\infty$, 
$\langle u(n)\rangle_{n=0}^\infty$, and $\langle v(n)\rangle_{n=0}^\infty$,
be as in the construction of ${\bf D}$. Define $k(0)=0$ and for each
$n<\omega$, let $k(n+1)=k(n)+v(n)$. (Then any row of ${\bf D}$ has in
columns $k(n),k(n)+1,\ldots,k(n)+v(n)-1$ either all $0$'s or a row of
$B_n$.)

Pick $\delta\in\ben$ such that for all $i\in\ohat{m-1}$ and
all $j\geq k(\delta)$, $c_{i,j}=0$.  Let $N$ be the restriction of $A$ to
columns $0,1,\ldots,k(\delta)-1$.  Let $M$ be a finite matrix whose 
rows are the nonzero rows of $N$ without repetition.  Then $M$ is a finite
IPR matrix since each row of $M$ followed by all $0$'s is a row of $A$.
So $M=B_l$ for some $l\in\omega$.  Note that $v(l)=k(\delta)$.

Choose $f:\{\delta,\delta+1,\ldots\}\inject \ben\setminus\{0\}$
so that for each $n\geq\delta$, the rows of $B_n$ are contained in the
rows of $B_{f(n)}$ and $v\big(f(n)\big)=v(n)$.

Now let $\vec x\in\ben^{\omega}$. We shall define $\vec y$ so that
the set of entries of $A\vec y$ are contained in the set of entries
of ${\bf D}\vec x$.  For $i\in\ohat{k(\delta)-1}$, let $y_i=x_{k(l)+i}$.
For $n\geq\delta$ and $i\in\ohat{v(n)-1}$, let $y_{k(n)+i}=x_{k(f(n))+i}$.

To see that the set of entries of $A\vec y$ are contained in the set of entries
of ${\bf D}\vec x$, let $\vec r$ be a row of $A$.  Define a row $\vec s$ of ${\bf D}$ as follows.
 For $i\in\ohat{k(\delta)-1}$, let $s_{k(l)+i}=r_i$.
For $n\geq\delta$ and $i\in\ohat{v(n)-1}$, let $s_{k(f(n))+i}=r_{k(n)+i}$.
If $n\in\omega\setminus(f[\{\delta,\delta+1,\ldots\}]\cup\{l\})$ and $i\in\ohat{v(n)-1}$,
then $s_{k(n)+i}=0$.
Then $\vec r\cdot\vec y=\vec s\cdot\vec x$.\end{proof}

\begin{conjecture}\label{qimmax} The system ${\bf D}$ is image maximal.
\end{conjecture}

The DH-matrix ${\bf D}$ seems a good candidate for a universal centrally IPR
matrix.  It trivially image dominates any finite IPR matrix.  
By Theorem \ref{DHScentral} it is strongly centrally IPR.
Therefore, if ${\bf D}$ image dominates a matrix $A$, it is immediate that
$A$ is centrally IPR. We see now, however, that
$A$ need not be strongly centrally IPR.

\begin{theorem}\label{notstrong} Let $A$ be any strongly centrally IPR
matrix and let 
$$B=\left(\begin{array}{cc}1&0\\ 3&-1\\ 5&-2\\ \vdots&\vdots\end{array}\right)\,.$$
Then $B$ is not strongly centrally IPR and $A$ image dominates $B$.
\end{theorem}

\begin{proof} By \cite[Theorem 2.11]{HLSa} $B$ is not strongly centrally IPR. 
To see that $A$ image dominates $B$, let $a$ be any element of an image of $A$.
Let $y_0=a$ and $y_1=2a$.  Then $Im(B\vec y)=\{a\}$.\end{proof}

One might hope (and we did) that any centrally IPR matrix is image dominated
by ${\bf D}$, or at least that any strongly centrally IPR matrix is image
dominated by ${\bf D}$.  (We knew that no Milliken-Taylor matrix which is not essentially a multiple
of ${\bf F}$ is image dominated by ${\bf D}$.)  We shall see that this fails.
To see it, we shall need another version of a DH-matrix (which is closer to the original in \cite{DH}).
The next definition differs from the description in Section \ref{secbackground} in that
here the entries are required to be non negative.

\begin{definition}\label{defmpc} Let $(m,p,c)\in\ben^3$.  A matrix $A$ is an $(m,p,c)$-matrix
if and only if $A$ is a first entries matrix with $m$ columns, all first entries are equal
to $c$, all entries of $A$ are in $\ohat{p}$, and $A$ contains all rows possible subject
to these restrictions.\end{definition} 

\begin{lemma}\label{mpcgood} Let $u,v\in\ben$ and let $A$ be a $u\times v$ matrix with
entries from $\beq$.  Then $A$ is IPR if and only if there
exist $(m,p,c)\in\ben^3$ such that for all $p'\geq p$, every
$(m,p',c)$-matrix $B$, and every $\vec y\in
\ben^m$, there exists $\vec x\in\ben^v$ such that $Im(A\vec x)\subseteq Im(B\vec y)$.\end{lemma}

\begin{proof} Since $(m,p,c)$-matrices are first entries matrices, the sufficiency
is immediate.  So assume that $A$ is IPR.
Pick by Theorem \ref{charipr}(c) $m\in \ben$, a $u\times m$ matrix
$E$ with entries from $\omega$, and $c\in\ben$ such that 
$E$ satisfies the first entries condition, $c$ is the only first entry 
of $E$, and given
any $\vec y\in\ben^m$ there is some $\vec x\in\ben^v$ with
$A\vec x=E\vec y$.  Let $p$ be the maximum of all of the entries of
$E$, let $p'\geq p$, and let $B$ be an $(m,p',c)$-matrix. Let $\vec y\in \ben^m$ be given
and pick $\vec x\in\ben^v$ such that $A\vec x=E\vec y$.
Then $Im(A\vec x)=Im(E\vec y)\subseteq Im(B\vec y)$.\end{proof}

Now we define our second version of a DH-matrix.
First fix an enumeration $\langle B'_n\rangle_{n=0}^\infty$
of the $(m,p,c)$-matrices where each $B'_n$ is an $\big(m(n),p(n),c(n)\big)$-matrix.
For each $i\in\ben$, let $\vec 0_i$ be the $0$ vector with $i$ entries.
Let ${\bf D'}$ be an $\omega\times\omega$ matrix with all rows
of the form $\vec r_0\con\vec r_1\con\vec r_2\con\ldots$ where
each $\vec r_i$ is either $\vec 0_{m(i)}$ or is a row of $B'_i$, and all but
finitely many are $\vec 0_{m(i)}$.

\begin{theorem}\label{imequiv} The DH-matrices ${\bf D}$ and ${\bf D'}$ are image
equivalent.  That is, each image dominates the other.\end{theorem}

\begin{proof}  Let $\langle B_n\rangle_{n=0}^\infty$,
$\langle v(n)\rangle_{n=0}^\infty$, $\langle B'_n\rangle_{n=0}^\infty$,
and $\langle m(n)\rangle_{n=0}^\infty$ be as in the construction of 
${\bf D}$ and ${\bf D'}$.  Since 
each $B'_n$ is some $B_k$, the fact that ${\bf D}$ image dominates ${\bf D'}$
is immediate.  

We now show that ${\bf D'}$ image dominates ${\bf D}$.  Using
Lemma \ref{mpcgood}, inductively define $f:\omega\inject\omega$
such that for every $\vec y\in
\ben^{m(f(n))}$, there exists $\vec x\in\ben^{v(n)}$ such that $Im(B_n\vec x)\subseteq 
Im(B'_{f(n)}\vec y)$.  

Inductively define $k(n)$ and $l(n)$ for $n\in\omega$ by $k(0)=l(0)=0$, 
and for $n\in\omega$, $k(n+1)=k(n)+v(n)$ and $l(n+1)=l(n)+m(n)$.
To see that ${\bf D'}$ image dominates ${\bf D}$, let $\vec w\in
\ben^\omega$ be given.  For $n\in\omega$, define
$\vec y_n\in\ben^{m(f(n))}$ by, for
$i\in\ohat{m\big(f(n)\big)-1}$, $y_{n,i}=w_{l(f(n))+i}$, and 
pick $\vec x_n\in\ben^{v(n)}$ such that $Im(B_n\vec x_n)\subseteq
Im(B'_{f(n)}\vec y_n)$.  Define $\vec z\in\ben^\omega$ by,
for $n\in\omega$ and $i\in\ohat{v(n)-1}$,
$z_{k(n)+i}=x_{n,i}$.  Then as in the proof of Theorem
\ref{DHfmax}, one sees that $Im({\bf D}\vec z)\subseteq Im({\bf D'}\vec w)$.
\end{proof}

Let $\langle c_n\rangle_{n=1}^\infty$ be a sequence in
$\ben$ and let
$${\mathcal I}=\left(\begin{array}{ccccccccc}1&0&0&0&0&0&0&0&\ldots\\
0&1&0&0&0&0&0&0&\ldots\\
c_1&1&0&0&0&0&0&0&\ldots\\
0&0&1&0&0&0&0&0&\ldots\\
0&0&0&1&0&0&0&0&\ldots\\
c_2&0&1&1&0&0&0&0&\ldots\\
0&0&0&0&1&0&0&0&\ldots\\
0&0&0&0&0&1&0&0&\ldots\\
0&0&0&0&0&0&1&0&\ldots\\
c_3&0&0&0&1&1&1&0&\ldots\\
\vdots&\vdots&\vdots&\vdots&\vdots&\vdots&\vdots&\vdots&\ddots\end{array}
\right)\,.$$

As in \cite[Theorem 16]{BHL}, one can show that ${\mathcal I}$ is IPR.
One can in fact show that it is strongly centrally IPR.
One can also show that if the sequence $\langle c_n\rangle_{n=1}^\infty$
is unbounded and $B$ is any matrix with the property that the 
entries of each column of $B$ are bounded, then $B$ does not image
dominate ${\mathcal I}$, and in particular ${\bf D'}$ does not
image dominate ${\mathcal I}$ and therefore, in view of Theorem \ref{imequiv},
${\bf D}$ does not image dominate ${\mathcal I}$.  We omit the verification
of these assertions because we have a much stronger example.

\begin{theorem}\label{Anodom} There is an $(\omega+\omega)\times\omega$
matrix $C$ with all entries from $\{0,1,2\}$ and all column sums equal to $3$
or $4$ which is strongly centrally IPR but is not image
dominated by ${\bf D}$. \end{theorem}

\begin{proof} Let $A$ be the $\omega\times\omega$ matrix such that, for
$i,j\in\omega$,
$$a_{i,j}=\left\{\begin{array}{cl}
0&\hbox{if }j<i\\
2&\hbox{if }j=i\\
0&\hbox{if }i<j<2^i\\
1&\hbox{if }2^i\leq j<2^{i+1}\\
0&\hbox{if }2^{i+1}\leq j\end{array}\right.$$
so that 
$$A=\left(\begin{array}{cccccccccc}
2&1&0&0&0&0&0&0&0&\ldots\\
0&2&1&1&0&0&0&0&0&\ldots\\
0&0&2&0&1&1&1&1&0&\ldots\\
\vdots&\vdots&\vdots&\vdots&\vdots&\vdots&\vdots&\vdots&\vdots&\ddots
\end{array}\right)$$
Let $I$ be the $\omega\times\omega$ identity matrix and
let $C=\left(\begin{array}{c}I\\ A\end{array}\right)$.
By \cite[Corollary 3.8]{BHLS} 
$C$ is strongly centrally IPR.  We shall show that
$C$ is not image dominated by ${\bf D}$ for which it suffices in view
of Theorem \ref{imequiv} to show that $C$ is not image dominated by ${\bf D'}$.
Let $\langle B'_n\rangle_{n=0}^\infty$ be as in the construction of
${\bf D'}$ where each $B'_n$ is an $\big(m(n),p(n),c(n)\big)$-matrix.

Define a function $f$ on $\omega$ by $f(0)=1$ and $f(i+1)=2^{f(i)+1}-1$.
We will define $\langle x_{i,j}\rangle_{j=0}^{m(i)-1}$ by induction on $i$. 

When we have defined $\langle x_{i,j}\rangle_{j=0}^{m(i)-1}$, we will let $S_i$ be the
set of entries of 
$$B'_i\left(\begin{array}{c}x_{i,0}\\
\vdots\\ x_{i,m(i)-1}\end{array}\right)$$
and let $M_i=\max\{S_0+S_1+\ldots+S_i\}$.

Pick $b_0>\max\{(2+2^{f(m(0))})c(0),p(0)\}$ and for $j\in\ohat{m(0)-1}$, let
$x_{0,j}=b_0^{j+1}$.  Let $i>0$ and assume we have chosen $\langle x_{i-1,j}\rangle_{j=0}^{m(i-1)-1}$,
$S_{i-1}$, and $M_{i-1}$.  Pick $b_i>\max\{(2+2^{f(m(i))})c(i),(2+2^{f(m(i))})M_{i-1},p(i)\}$ such that
$b_{i-1}$ divides $b_i$. For $j\in\ohat{m(i)-1}$, let $x_{i,j}=b_i^{j+1}$.

Notice that since $b_i>p(i)$ we have that expressions in 
$$B_i\left(\begin{array}{c}x_{i,0}\\
\vdots\\ x_{i,m(i)-1}\end{array}\right)$$
are unique.  That is, if $\vec r$ and $\vec s$ are rows of $B_i$ and 
$$\vec r\left(\begin{array}{c}x_{i,0}\\
\vdots\\ x_{i,m(i)-1}\end{array}\right)=\vec s\left(\begin{array}{c}x_{i,0}\\
\vdots\\ x_{i,m(i)-1}\end{array}\right)\,,$$
then $\vec r=\vec s$.

Notice that, given $y\in S_i$, there exist some $l\in\nhat{m(i)}$ and some
$d\in\omega$ such that $y=c(i)b_i^l+db_i^{l+1}$.

The construction being complete, let $$\vec z=\left(\begin{array}{c}x_{0,0}\\
\vdots\\ x_{0,m(0)-1}\\
x_{1,0}\\ \vdots\\ x_{1,m(1)-1}\\ \vdots\end{array}\right)$$
and let $J=Im({\bf D'}\vec z)$. 
Notice that $J=\bigcup_{F\in\pf(\omega)}\sum_{i\in F}S_i=FS(\langle S_i\rangle_{i=0}^\infty)$.

\begin{definition}\label{defSupp}
For $y\in J$, Supp$(y)$ is that
$F\in\pf(\omega)$ such that $y\in \sum_{i\in F}S_i$.
For $y\in J$ and $i\in\ben$, $\pi_i(y)=0$ if $i\notin\hbox{Supp}(y)$ and 
otherwise, $\pi_i(y)\in S_i$ such that $y=\sum_{i\in\hbox{\smallrm Supp}(y)}\,\pi_i(y)$.
\end{definition}

Given $i\in\omega$, we have that two elements of $S_{i+1}$ differ by at least
$b_{i+1}$ and $b_{i+1}>M_i$ so expressions in 
$\sum_{i\in\hbox{\smallrm Supp}(y)}S_i$ are unique and thus $\pi_i$ is well defined.

 We claim that there is no $\vec y\in\ben^\omega$
such that $Im(C\vec y)\subseteq J$, so suppose instead that we have
such $\vec y$. Let $q=\min\hbox{Supp}(y_0)$.  (Any other member of $\hbox{Supp}(y_0)$ would do
just as well, with no change in the proof.)

\begin{lemma}\label{pilem} Let $v\in \big\{0,1,\ldots,f\big(m(q)\big)\big\}$.  Then
$$\textstyle\pi_q(2y_v+\sum_{k=2^v}^{2^{v+1}-1}y_k)=2\pi_q(y_v)+\sum_{k=2^v}^{2^{v+1}-1}\pi_q(y_k)\,.$$
\end{lemma}

\begin{proof} Let $z=2y_v+\sum_{k=2^v}^{2^{v+1}-1}y_k$.  Then $z\in J$ so pick
$a_0\in\omega$ such that $z=a_0b_{q+1}+\pi_q(z)+\sum_{i=0}^{q-1}\pi_i(z)$. 
(Here $a_0b_{q+1}=\sum\{\pi_i(z):i\in \hbox{Supp}(z)$ and $i>q\}$ if 
$\{i\in\hbox{Supp}(z):i>q\}\neq\emp$.)

For each
$k\in\{v\}\cup\{2^v,2^v+1,\ldots,2^{v+1}-1\}$, pick $a_k\in\omega$ such that
$y_k=a_kb_{q+1}+\pi_q(y_k)+\sum_{i=0}^{q-1}\pi_i(y_k)$.
Then 
$$\begin{array}{rl}z=&\textstyle(2a_v+\sum_{k=2^v}^{2^{v+1}-1}a_k)b_{q+1}+{}\\
&\textstyle 2\pi_q(y_v)+\sum_{k=2^v}^{2^{v+1}-1}\pi_q(y_k)+{}\\
&\textstyle \sum_{i=0}^{q-1}\big(2\pi_i(y_v)+\sum_{k=2^v}^{2^{v+1}-1}\pi_i(y_k)\big)\,.\end{array}$$
Now $\sum_{i=0}^{q-1}\big(2\pi_i(y_v)+\sum_{k=2^v}^{2^{v+1}-1}\pi_i(y_k)\big)\leq (2^v+2)M_{q-1}\leq
(2^{f(m(q))}+2)M_{q-1}<b_q$.
And of course $\sum_{i=0}^{q-1}\pi_i(z)\leq M_{q-1}<b_q$.

Since also $b_q$ divides $a_0b_{q+1}+\pi_q(z)$ and $b_q$ divides
$$\textstyle(2a_v+\sum_{k=2^v}^{2^{v+1}-1}a_k)b_{q+1}+2\pi_q(y_v)+\sum_{k=2^v}^{2^{v+1}-1}\pi_q(y_k)$$
we have that 
$$\textstyle(2a_v+\sum_{k=2^v}^{2^{v+1}-1}a_k)b_{q+1}+2\pi_q(y_v)+\sum_{k=2^v}^{2^{v+1}-1}\pi_q(y_k)=
a_0b_{q+1}+\pi_q(z)\,.$$

Similarly $2\pi_q(y_v)+\sum_{k=2^v}^{2^{v+1}-1}\pi_q(y_k)<b_{q+1}$ and $\pi_q(z)<b_{q+1}$ so these
are equal as claimed. 
\end{proof}

\begin{lemma}\label{geqtwo}  Let $v\in\big\{0,1,\ldots,f\big(m(q)\big)\big\}$ such that $\pi_q(y_v)\neq 0$.
Pick $l\in\nhat{m(q)}$ and $d\in\omega$ such that
$\pi_q(y_v)=c(q)b^l+db_q^{l+1}$.  Then $l\geq 2$ and for some $i\in\{2^v,2^v+1,\ldots,2^{v+1}-1\}$,
some $l'\in\nhat{l-1}$, and some $d'\in\omega$, $\pi_q(y_i)=c(q)b^{l'}+d'b_q^{l'+1}$.
\end{lemma}

\begin{proof} Since $\pi_q(y_v)\neq 0$ we have by Lemma \ref{pilem} that 
$\pi_q(2y_v+\sum_{k=2^v}^{2^{v+1}-1}y_k)\neq 0$.  Pick $t\in\nhat{m(q)}$ and $e\in\omega$ such that
$\pi_q(2y_v+\sum_{k=2^v}^{2^{v+1}-1}y_k)=c(q)b_q^t+eb_q^{t+1}$.

Let $H=\{k\in\{2^v,2^v+1,\ldots,2^{v+1}-1\}:\pi_q(y_k)\neq 0\}$.  If $H=\emp$, then
by Lemma \ref{pilem}, $c(q)b_q^t+eb_q^{t+1}=2c(q)b_q^l+2db_q^{l+1}$ so, since $b_q>2c(q)$, we have
$t=l$ and $c(q)=2c(q)$, a contradiction.  So $H\neq\emp$.

For $k\in H$, pick $l_k\in\nhat{m(q)}$ and $d_k\in\omega$ such that 
$\pi_q(y_k)=c(q)b_q^{l_k}+d_kb_q^{l_k+1}$.  We need to show that some
$l_k<l$, so suppose instead that each $l_k\geq l$. We have by Lemma \ref{pilem} that
$$\textstyle c(q)b_q^t+eb_q^{t+1}=2c(q)b_q^l+2db_q^{l+1}+\sum_{k\in H}(c(q)b_q^{l_k}+d_kb_q^{l_k+1})\,.$$
If each $l_k>l$ we again conclude that $t=l$ and $c(q)=2c(q)$.  Let 
$K=\{k\in H:l_k=l\}$ and let $\delta=|K|$.  Then we get
$$\textstyle 2c(q)b_q^l+2db_q^{l+1}+\sum_{k\in H}(c(q)b_q^{l_k}+d_kb_q^{l_k+1})=
(2+\delta)c(q)b_q^l+\alpha b_q^{l+1}$$ for some $\alpha\in\omega$.
But $\delta\leq 2^v\leq 2^{f(m(q))}$ so
$(2+\delta)c(q)\leq (2+2^{f(m(q))})c(q)<b_q$ so $t=l$ and $c(q)=(2+\delta)c(q)$,
a contradiction.\end{proof}

We are now ready to complete the proof of the theorem.
Pick $l_0\in\{1,2,\ldots,\break 
m(q)\}$ and $d_0\in\omega$ such that $\pi_q(y_0)=c(q)b_q^{l_0}+d_0b_q^{l_0+1}$.
By Lemma \ref{geqtwo}, $l_0\geq 2$ and we may pick $i(1)=1$,
$d_1\in\omega$, and $l_1\in\nhat{l_0-1}$ such that $\pi_q(y_{i(1)})=c(q)b_q^{l_1}+d_1b_q^{l_1+1}$.

Given $t\geq 1$, $i(t)$, $l_t$, and $d_t$ such that
$i(t)\leq f(t)\leq f(m(q))$ and $\pi_q(y_{i(t)})=c(q)b_q^{l_t}+d_tb_q^{l_t+1}$, pick
by Lemma \ref{geqtwo}, $i(t+1)\in \{2^{i(t)},2^{i(t)}+1,\ldots,2^{i(t)+1}-1\}$,
$l_{t+1}<l_t$, and $d_{t+1}\in\omega$ such that 
$\pi_q(y_{i(t+1)})=c(q)b_q^{l_{t+1}}+d_{t+1}b_q^{l_{t+1}+1}$.
Then $i(t+1)\leq 2^{i(t)+1}-1\leq 2^{f(t)+1}-1=f(t+1)$.
Also $m(q)\geq l_0>l_1>\ldots>l_{t+1}$ so $m(q)>t+1$ and thus
$i(t+1)\leq f(t+1)<f(m(q))$.  When $t+1=m(q)$ we have a contradiction.\end{proof}

Note that the matrix of Theorem \ref{Anodom} has unbounded row sums (as does ${\bf D'}$).

\begin{question}\label{qbdedrows} Let $A$ be an $\omega\times\omega$ centrally IPR
matrix with the property that\break $\{\sum_{j=0}^\infty|a_{i,j}|:i<\omega\}$ is bounded.
Must $A$ be image dominated by ${\bf D}$?\end{question}

\section{Translates of MT-Matrices}

As we saw in Theorem \ref{MTsep}, if $k\in\ben$, 
$\vec a=\langle a_0,a_1,\ldots,a_k\rangle$ and 
$M$ is an $MT(\vec a)$ matrix, then 
$\left(\begin{array}{cc} M&{\bf O}\\{\bf O}& {\bf F}\end{array}\right)$
is not IPR.  We shall see in
Theorem \ref{MTtran}, that 
$\left(\begin{array}{cc} \overline 1&M\\ \overline 0& {\bf F}\end{array}\right)$
is partition regular, where $\overline 1$ and $\overline 0$ are the constant
length $\omega$ column vectors.  That is, given any finite colouring of $\ben$, there
must exist a sequence $\vec x=\langle x_n\rangle_{n=0}^\infty$ and $b\in\ben$
such that $FS(\vec x)\cup \big(b+MT(\vec a,\vec x)\big)$ is monochromatic.

Given $a\in\bez$ and $p\in\beta\ben$, by $ap$ we mean the product in
$(\beta\bez,\cdot)$.  (If $p\in\ben^*$ it is not even true that
$2p=p+p$.) If $A\subseteq \bez$, then $A\in ap$ if and only if
$a^{-1}A\in p$.  Since $\ben\in p$, then $A\in ap$ if and only if
$\{x\in \ben:ax\in A\}\in p$.

The basic algebraic property of $\beta\ben$ used in the following lemma is that
$p+\beta\ben+p$ is a group in $\beta\ben$ whenever $p$ is an idempotent in the smallest ideal of
$\beta\ben$.

\begin{lemma}\label{lemtran} Let $k\in\ben$ and let
$\vec a=\langle a_0,a_1,\ldots,a_k\rangle$ be a compressed sequence
in $\bez\setminus\{0\}$ with $a_k=1$.  Let $p$ be a minimal idempotent
in $\beta\ben$ and let $A\in p$.  There exists $b\in \ben$ such that
$-b+A\in a_0p+a_1p+\ldots+a_kp$.
\end{lemma}

\begin{proof} By \cite[Exercise 4.3.5]{HS}, $\ben^*$ is a left ideal
of $(\beta\bez,+)$, so $\beta \ben+p\subseteq \beta \bez+p=
\beta\bez+p+p\subseteq\ben^*+p\subseteq\beta \ben+p$. Therefore
$$\begin{array}{rl}p+a_0p+a_1p+\ldots+a_kp&\hskip -4 pt=p+(a_0p+\ldots+a_{k-1}p)
+p\\
&\hskip -4 pt\in p+\beta\bez+p\\
&\hskip -4 pt=p+\beta\ben+p\end{array}$$ and, since $p$ is minimal,
$p+\beta\ben+p$ is a group.  Pick $q\in p+\beta\ben+p$ such that
$q+p+a_0p+a_1p+\ldots+a_kp=p$.  Since $q+p=q$, 
$A\in q+a_0p+a_1p+\ldots+a_kp$ so
$\{x\in\ben:-x+A\in a_0p+a_1p+\ldots+a_kp\}\in q$.
Pick $b\in \{x\in\ben:-x+A\in a_0p+a_1p+\ldots+a_kp\}$.\end{proof}

Before giving the proof of Theorem \ref{MTtran} in the general case, we shall first 
give the proof for a  simple
special case. We should like the reader to understand the simple idea underlying the proof,
before having to read the rather daunting details of the general proof.

\begin{theorem}\label{simplecase} Let $\vec a=\langle 2,1\rangle$, let $p$ be a minimal
idempotent in $\beta\ben$ and let $A\in p$. Then there exist
$b\in\ben$ and a sequence $\langle x_n\rangle_{n=0}^{\infty}$ in $\ben$ such that
$FS(\langle x_n\rangle_{n=0}^\infty)\subseteq A$ and 
 $b+MT(\vec a\langle x_n\rangle_{n=0}^\infty)\subseteq A$.
\end{theorem}

\begin{proof} 
By Lemma \ref{lemtran} We can choose $b\in \ben$ such that $-b+A\in 2p+p$.
Given $B\in p$, let $B^\star=\{x\in B:-x+B\in p\}$.  By
\cite[Lemma 4.14]{HS}, $B^{\star}\in p$ and, if $x\in B^{\star}$, then 
$-x+B^\star\in p$. 

We put $B=\{x\in \ben:2x+p\in \overline{-b+A}\}$ and, for each
$x\in B$, we put $B(x)=\{y\in \ben:2x+y\in -b+A\}$. We observe that $B$ and $B(x)$
are members of $p$.

We shall inductively construct a sequence $\langle x_n\rangle_{n=0}^{\infty}$ in $\ben$ 
 such that $$FS(\langle x_n\rangle_{n=0}^{\infty})\subseteq A^{\star}\cap B^{\star}$$ and, whenever
$F,G\in\pf(\ben)$ and $F<G$,  then $\sum_{n\in G}x_n\in B(\sum_{m\in F}x_m)^{\star}$.

We choose any $x_0\in B^{\star}$. We then assume that $r\geq 0$ and that we have
chosen a sequence $\langle x_0,x_1,x_2\ldots,x_r\rangle$ so that  $FS(\langle x_n\rangle_{n=0}^r)\subseteq A^{\star}\cap B^{\star}$,  and, whenever
$F,G\in\pf(\{1,2,\ldots,r\})$ and $F<G$,
then  $\sum_{n\in G}x_n\in B(\sum_{m\in F}x_m)^{\star}$.

If $F\in\pf(\ohat{r})$, then the following sets are all members of $p$:
 $$\textstyle -\sum_{m\in F}x_m+A^\star\,,\,  -\sum_{m\in F}x_m+B^\star\hbox{ and }
B(\sum_{m\in F}x_m)^\star\,.$$ Furthermore, if $G\in\pf(\{1,2,\ldots,r\})$ and $F<G$, then
 $$\textstyle -\sum_{n\in G}x_n+B(\sum_{m\in F}x_m)^{\star}\in p\,.$$
So all the sets of this form have a non-empty intersection with $A^{\star}\cap B^{\star}$,
and we can choose an element $x_{r+1}\in A^{\star}\cap B^{\star}$  which is in all these sets.
It is then routine to check that our inductive hypotheses extend to the sequence
$\langle x_0,x_1,\ldots,x_r,x_{r+1}\rangle$.
\end{proof}

Note that, in the following theorem, if one wishes, one can
let $\langle \vec a_i\rangle_{i=0}^\infty$ enumerate all of the
compressed sequences in $\bez\setminus\{0\}$ with final term equal to $1$. 
The proof of the following theorem is based on the proof of
\cite[Theorem 17.31]{HS}.  The reader is referred to that 
proof for details involved in verifying the induction hypotheses.

\begin{theorem}\label{MTtran} For each $i<\omega$, let $k(i)\in\ben$ and
let $\vec a_i=\langle a_{i,0},a_{i,1},\ldots,a_{i,k(i)}\rangle$
be a compressed sequence in $\bez\setminus\{0\}$ with $a_{i,k(i)}=1$.
Let $p$ be a minimal idempotent in $\beta \ben$ and let $A\in p$.
There exists  sequences $\langle b_n\rangle_{n=0}^\infty$
and $\langle x_n\rangle_{n=0}^\infty$ in $\ben$ such that 
$FS(\langle x_n\rangle_{n=0}^\infty)\subseteq A$ and 
for each $i\in\omega$, $b_i+MT(\vec a_i,\langle x_n\rangle_{n=i}^\infty)
\subseteq A$.
\end{theorem}

\begin{proof} For each $i\in\omega$, pick by Lemma \ref{lemtran}, $b_i\in
\ben$ such that $-b_i+A\in a_{i,0}p+a_{i,1}p+\ldots+a_{i,k(i)}p$. 
Given $B\in p$, let $B^\star=\{x\in B:-x+B\in p\}$.  By
\cite[Lemma 4.14]{HS}, if $x\in B^\star$, then 
$-x+B^\star\in p$.

Let $B_0=A\cap\{x\in\ben:-a_{0,0}x+(-b_0+A)\in a_{0,1}p+\ldots+a_{0,k(0)}p\}$ and
pick $x_0\in B_0^\star$.

Now let $n\in\omega$ and assume that we have chosen $\langle x_j\rangle_{j=0}^n$
 in $\ben$
and $\langle B_j\rangle_{j=0}^n$ in $p$ so that for each $r\in\ohat{n}$ the following
induction hypotheses hold.
\begin{itemize}
\item[(I)] If $\emp\neq F\subseteq\ohat{r}$ and $i=\min F$, then
$\sum_{t\in F}x_t\in B_i^\star$ and\hfill  $B_i\subseteq \{x\in\ben:
-a_{i,0}x+(-b_i+A)\in a_{i,1}p+\ldots+a_{i,k(i)}p\}$.
\item[(II)] If $r<n$, then $B_{r+1}\subseteq B_r$.
\item[(III)] If $i\in \ohat{r}$, $l\in\ohat{k(i)-1}$,\hfill 
$F_0,F_1,\ldots,F_l\in\pf(\{i,i+1,\ldots,r\})$, and $F_0<F_1<\ldots<F_l$, then 
\hfill 
$-\sum_{j=0}^la_{i,j}\sum_{t\in F_j}x_t+(-b_i+A)\in a_{i,l+1}p+\ldots+a_{i,k(i)}
p$.
\item[(IV)] If $i\in \ohat{r}$, $F_0,F_1,\ldots,F_{k(i)-1}\in
\pf(\{i,i+1,\ldots,r\})$,\hfill  $F_0<F_1<\ldots<F_{k(i)-1}$, and $r<n$, then\hfill 
$B_{r+1}\subseteq -\sum_{j=0}^{k(i)-1}a_{i,j}\sum_{t\in F_j}x_t+(-b_i+A)$.
\item[(V)] If $i\in \ohat{r}$, $l\in\ohat{k(i)-2}$,\hfill 
 $F_0,F_1,\ldots,F_l\in \pf(\{i,i+1,\ldots,r\})$, $F_0<F_1<\ldots<F_l$, and $r<n$, then
$$\begin{array}{rl} B_{r+1}\subseteq\{x\in\ben:&\textstyle\hskip -5 pt
-a_{i,l+1}x+\big(-\sum_{j=0}^la_{i,j}\sum_{t\in F_j}x_t+(-b_i+A)\big)\in{}\\
&\hskip -5 pt a_{i,l+2}p+\ldots+a_{i,k(i)}p\}\,.\end{array}$$
\end{itemize}

For $i\in\ohat{n}$ and $l\in\ohat{k(i)-1}$, let 
$$\begin{array}{rl}{\mathcal F}_{i,l}=\{(F_0,F_1,\ldots,F_l):&\hskip - 5 pt 
F_0,F_1,\ldots,F_l\in\pf(\{i,i+1,\ldots,n\})\hbox{ and }\\
&\hskip -5 pt F_0<F_1<\ldots<F_l\}\,.\end{array}$$
(Of course, if $l>n-i$, then ${\mathcal F}_{i,l}=\emp$.) For $m\in\ohat{n}$, let
$E_m=\{\sum_{t\in F}x_t:\emp\neq F\subseteq\ohat{n}$ and $\min F=m\}$.

In the definition of $B_{n+1}$ below, we use the convention that
$\bigcap\emp=\ben$.  So, for example, if $i\in \ohat{n}$, $l\in\ohat{k(i)-2}$, 
and
${\mathcal F}_{i,l}=\emp$, then one ignores the term
$$\begin{array}{rl}\{x\in\ben:
&\hskip -6 pt\textstyle -a_{i,l+1}x+\big(-\sum_{j=0}^la_{i,j}\sum_{t\in F_j}x_t
+(-b_i+A)\big)\in{}\\
&\hskip -6 pt a_{i,l+2}p+\ldots+a_{i,k(i)}p\}\,.\end{array}$$

Let
$$\begin{array}{rl}
B_{n+1}=&\hskip -6 pt
\{x\in\ben:-a_{n+1,0}x+(-b_{n+1}+A)\in a_{n+1,1}p+\ldots+a_{n+1,k(n+1)}p\}\cap{}
\\
&\hskip -6 pt\textstyle B_n\cap\bigcap_{m=0}^n\bigcap_{c\in E_m}(-c+B_m^\star)
\cap{}\\
&\hskip -6 pt\textstyle\bigcap_{i=0}^n\bigcap_{(F_0,\ldots,F_{k(i)-1})\in 
{\mathcal F}_{i,k(i)-1}}
-\sum_{j=0}^{k(i)-1}a_{i,j}\sum_{t\in F_j}x_t+(-b_i+A)\cap{}\\
&\hskip -6 pt 
\bigcap_{i=0}^n\bigcap_{l=0}^{k(i)-2}
\bigcap_{(F_0,\ldots,F_l)\in {\mathcal F}_{i,l}}\\
&\hskip -12 pt\begin{array}{rl}\{x\in\ben:
&\hskip -6 pt\textstyle -a_{i,l+1}x+\big(-\sum_{j=0}^la_{i,j}\sum_{t\in F_j}x_t
+(-b_i+A)\big)\in{}\\
&\hskip -6 pt a_{i,l+2}p+\ldots+a_{i,k(i)}p\}\,.\end{array}\end{array}$$
Then $B_{n+1}\in p$.  Pick $x_{n+1}\in B_{n+1}^\star$.

The induction being complete, we have that 
$FS(\langle x_n\rangle_{n=0}^\infty)\subseteq A$ by hypotheses (I) and (II) 
and the fact that $B_0\subseteq A$. Finally, let $i\in\omega$ and 
let $F_0,F_1,\ldots,F_{k(i)}$ be given in $\pf(\{i,i+1,\ldots\})$ such that 
$F_0<F_1<\ldots<F_{k(i)}$.  Let $n=\min F_{k(i)}$ and $m=\max F_{k(i)-1}$.
By hypotheses (I) and (II), $\sum_{t\in F_{k(i)}}x_t\in B_n\subseteq B_{m+1}$ so
 by hypothesis
(IV), $\sum_{t\in F_{k(i)}}x_t\in -\sum_{j=0}^{k(i)-1}a_{i,j}\sum_{t\in F_j}x_t+
(-b_i+A)$. Thus,
since $a_{i,k(i)}=1$, $b_i+\sum_{j=0}^{k(i)}a_{i,j}\sum_{t\in F_j}x_t+ A$.
\end{proof}

\begin{corollary}\label{MTextndmat} Let $m\in\omega$ and for 
each $i\in\ohat{m}$, let $k(i)\in\ben$,
let $\vec a_i=\langle a_{i,0},a_{i,1},\ldots,a_{i,k(i)}\rangle$
be a compressed sequence in $\bez\setminus\{0\}$ with $a_{i,k(i)}=1$,
and let $M_i$ be an $MT(\vec a_i)$-matrix. Let $\overline 0$ and $\overline 1$ be the length
$\omega$ constant vectors. Then
$$B=\left(\begin{array}{ccccc}\overline 1&\overline 0&\ldots&\overline 0&M_0\\
\overline 0&\overline 1&\ldots&\overline 0&M_1\\
\vdots&\vdots&\ddots&\vdots&\vdots\\
\overline 0&\overline 0&\ldots&\overline 1&M_m\\
\overline 0&\overline 0&\ldots&\overline 0&{\bf F}
\end{array}\right)$$
is centrally IPR.\end{corollary}

\begin{proof} Let $A$ be a central set and pick a minimal idempotent
$p$ such that $A\in p$. For $i>m$ let $\vec a_i=\langle 2,1\rangle$ 
(or any other reasonable choice)
and let $\langle b_n\rangle_{n=0}^\infty$
and $\langle x_n\rangle_{n=0}^\infty$ be as guaranteed by Theorem \ref{MTtran}. 

For $n<\omega$, let $y_n=x_{m+n}$.  Then all entries of 
$$B\left(\begin{array}{c}b_0\\ b_1\\ \vdots\\ b_m\\
\vec y\end{array}\right)$$
are in $A$. \end{proof}

What Theorem \ref{MTtran} is telling us is that, if we are allowed to add new
variables (to represent the `translation') then {\bf F} is very far from being
maximal. This motivates the following definition. We say that 
an IPR matrix $A$ is {\it universally image maximal\/}
provided that whenever $B$ is an IPR matrix that
image dominates $A$, then $A$ image dominates $B$

Is {\bf D} universally image maximal? One might hope that the answer is yes,
but it turns out that, similarly to Theorem \ref{MTtran}, one can actually
extend {\bf D} by a translate of what one might call a  `DHMT' system. 

\begin{definition}\label{defMTY} Let $k\in\ben$, let $\vec a=\langle a_0,a_1,\ldots,a_k\rangle$ and for each $n<\omega$ let 
$Y_n\in\pf(\beq)$.  Then $MT(\vec a,\langle Y_n\rangle_{n=0}^\infty)=
\{\sum_{i=0}^ka_i\sum_{t\in F_i}x_t:F_0,F_1,\ldots,F_k\in \pf(\omega)\,,\,
F_0<F_1<\ldots<F_k\hbox{ and }x\in\bigtimes_{t\in\bigcup_{i=0}^k F_i}Y_t\}$.
\end{definition}

Fix an enumeration $\langle B_n\rangle_{n=0}^\infty$ of the 
finite IPR matrices with rational entries.  
For each $n$, assume that $B_n$ is a $u(n)\times v(n)$ matrix.

\begin{theorem} Let $p$ be a minimal idempotent in $\beta\ben$ and let 
$A\in p$.  There exist $b\in\ben$ and a sequence $\langle Y_n\rangle_{n=0}^\infty$
in $\pf(\ben)$ such that each $Y_n$ is the set of entries of an image of $B_n$ and
$FS(\langle Y_n\rangle_{n=0}^\infty)\cup b+MT(\langle 2,1\rangle,\langle Y_n\rangle_{n=0}^\infty)
\subseteq A$.
\end{theorem}

\begin{proof} As in the proof of Lemma \ref{lemtran}, pick $q\in\beta\ben$ 
such that $p=q+2p+p$ and pick
$b\in\ben$ such that $-b+A^\star\in 2p+p$.  Let
$D=\{x\in A^\star:-2x+(-b+A^\star)\in p\}$.  Then $D\in p$.

Choose $\vec x(0)\in\ben^{v(0)}$ such that, letting $Y_0$ be the set of
entries of $B_0\vec x(0)$, we have $Y_0\subseteq D^\star$.

Inductively let $n\in\omega$ and assume that we have chosen $\vec x(k)\in \ben^{v(k)}$
such that, letting $Y_k$ be the set of
entries of $B_k\vec x(k)$, we have that

\begin{itemize}
\item[(1)] $FS(\langle Y_k\rangle_{k=0}^n)\subseteq D^\star$ and
\item[(2)] if $n>0$, then $MT(\langle 2,1\rangle,\langle Y_k\rangle_{k=0}^n)
\subseteq -b+A^\star$.
\end{itemize}

Now, if $x\in FS(\langle Y_k\rangle_{k=0}^n)$, then
$-x+D^\star\in p$ and $-2x+(-b+A^\star)\in p$.  Also, if
$n>0$ and $x\in MT(\langle 2,1\rangle,\langle Y_k\rangle_{k=0}^n)$, 
then $x\in (-b+A^\star)$ so $b+x\in A^\star$ and thus
$-(b+x)+A^\star\in p$.

Pick $\vec x(n+1)\in\ben^{v(n+1)}$
such that 
$$\begin{array}{rl}Y_{n+1}\subseteq&\textstyle D^\star\cap\bigcap_{x\in FS(\langle Y_k\rangle_{k=0}^n)}
\big((-x+D^\star)\cap (-2x+(-b+A^\star))\big)\hfill\\
&\textstyle\cap\bigcap_{x\in MT(\langle 2,1\rangle,\langle Y_k\rangle_{k=0}^n)}(-(b+x)+A^\star))\,.\end{array}$$

To see that $FS(\langle Y_k\rangle_{k=0}^{n+1})\subseteq D^\star$, let $\emp\neq F\subseteq
\ohat{n+1}$ and let $x\in\bigtimes_{t\in F}Y_t$. If $n+1\notin F$, we have that 
$\sum_{t\in F}x_t\in D^\star$ by hypothesis (1). If $F=\{n+1\}$, then $x_{n+1}\in Y_{n+1}
\subseteq D^\star$.  So assume that $\{n+1\}\subsetneq F$ and let $F'=F\setminus\{n+1\}$.
Then $x_{n+1}\in -(\sum_{t\in F'}x_t)+D^\star$ so
$\sum_{t\in F}x_t\in D^\star$.

To verify that $MT(\langle 2,1\rangle,\langle Y_k\rangle_{k=0}^{n+1})
\subseteq -b+A^\star$, let
$F,H\in\pf(\nhat{n+1})$ such that $\max F<\min H$ and let $x\in\bigtimes_{k\in F\cup H}\, Y_k$.
If $\max H<n+1$ the conclusion holds by the hypothesis (2), so assume that
$n+1\in H$. If $H=\{n+1\}$, then $x_{n+1}\in \big(-2\sum_{t\in F}x_t+(-b+A^\star)\big)$
so $\sum_{t\in F}2x_t+x_{n+1}\in -b+A^\star$.

Now assume that $\{n+1\}\subsetneq H$ and let $H'=H\setminus\{n+1\}$. 
Then $x_{n+1}\in -(b+\sum_{t\in F}2x_t+\sum_{t\in H'}x_t)+A^\star$
so $\sum_{t\in F}2x_t+\sum_{t\in H}x_t\in -b+A^\star$ as required.\qed
\end{proof}

We remark that the analogue of Theorem \ref{MTtran} wherein
$\langle x_n\rangle_{n=0}^\infty$ is replaced by 
$\langle Y_n\rangle_{n=0}^\infty$ remains valid with essentially the same
proof.

We do not know any examples of universally image maximal systems.

\begin{question} Does there exist a universally image maximal
matrix?
\end{question}

\bibliographystyle{plain}

\end{document}